\documentclass[12pt]{amsart}

\usepackage{amssymb,amsmath,amsthm,amsfonts,amscd,bbm}
\usepackage{enumitem}
\usepackage{pdflscape}% rotating the page
\usepackage[paper=portrait,pagesize]{typearea}
\usepackage{caption}
\usepackage{bm}

\usepackage{cite}

\usepackage{geometry} % change margin
\usepackage[all]{xy}
\usepackage{hyperref}
\usepackage{xcolor}
\usepackage{graphicx}
\usepackage[outdir=./]{epstopdf}
\usepackage{enumitem}
\usepackage{tensor}
\usepackage{multicol}
\usepackage{mathtools}
\usepackage{tocvsec2}
\usepackage{bbm}
\usepackage{longtable}
\usepackage{fullpage}

\newcommand{\doi}[1]{\href{https://dx.doi.org/#1}{{\small DOI:#1}}}

\newcommand{\googlebooks}[1]{(preview at \href{https://books.google.com/books?id=#1}{google books})}

\newcommand{\numdam}[1]{}
\usepackage{mathrsfs}
\usepackage{mathabx}

\DeclareMathAlphabet{\mathpzc}{OT1}{pzc}{m}{it}

\def\semicolon{;}
\def\applytolist#1{
    \expandafter\def\csname multi#1\endcsname##1{
        \def\multiack{##1}\ifx\multiack\semicolon
            \def\next{\relax}
        \else
            \csname #1\endcsname{##1}
            \def\next{\csname multi#1\endcsname}
        \fi
        \next}
    \csname multi#1\endcsname}

\def\calc#1{\expandafter\def\csname c#1\endcsname{{\mathcal #1}}}
\applytolist{calc}QWERTYUIOPLKJHGFDSAZXCVBNM;
\def\bbc#1{\expandafter\def\csname bb#1\endcsname{{\mathbb #1}}}
\applytolist{bbc}QWERTYUIOPLKJHGFDSAZXCVBNM;
\def\bfc#1{\expandafter\def\csname bf#1\endcsname{{\mathbf #1}}}
\applytolist{bfc}QWERTYUIOPLKJHGFDSAZXCVBNM;
\def\sfc#1{\expandafter\def\csname s#1\endcsname{{\sf #1}}}
\applytolist{sfc}QWERTYUIOPLKJHGFDSAZXCVBNM;
\def\fc#1{\expandafter\def\csname f#1\endcsname{{\mathfrak #1}}}
\applytolist{fc}QWERTYUIOPLKJHGFDSAZXCVBNM;
\def\rmc#1{\expandafter\def\csname rm#1\endcsname{{\mathrm #1}}}
\applytolist{rmc}QWERTYUIOPLKJHGFDSAZXCVBNM;

\usepackage{tikz}
\usepackage{tikz-cd}

\makeatletter
\def\fixtikzforbreqn#1#2{%
  \protected\edef#1{\noexpand\ifmmode\mathchar\the\mathcode`#2 \noexpand\else#2\noexpand\fi}%
}
\fixtikzforbreqn\tikz@nonactivesemicolon;
\fixtikzforbreqn\tikz@nonactivecolon:
\fixtikzforbreqn\tikz@nonactivebar|
\fixtikzforbreqn\tikz@nonactiveexlmark!
\makeatother

\usepackage{breqn}   % Fix : | ; ! in tikzcd. And must put makealetter-other between tikzcd package and breqn package

\usetikzlibrary{arrows,backgrounds,patterns.meta}
\usetikzlibrary{positioning,shadings,cd}
\usetikzlibrary{shapes}
\usetikzlibrary{backgrounds}
\usetikzlibrary{decorations,decorations.pathreplacing,decorations.markings,decorations.pathmorphing}
\usetikzlibrary{fit,calc,through}
\usetikzlibrary{external}
\usetikzlibrary{arrows}
\tikzset{vertex/.style = {shape=circle,draw,fill=black,inner sep=0pt,minimum size=5pt}}
\tikzset{edge/.style = {->,> = latex', bend right}}
\tikzset{
	super thick/.style={line width=3pt}
}
\tikzset{
    quadruple/.style args={[#1] in [#2] in [#3] in [#4]}{
        #1,preaction={preaction={preaction={draw,#4},draw,#3}, draw,#2}
    }
}
\tikzstyle{shaded}=[fill=red!10!blue!20!gray!30!white]
\tikzstyle{unshaded}=[fill=white]
\tikzstyle{empty box}=[circle, draw, thick, fill=white, opaque, inner sep=2mm]
\tikzstyle{annular}=[scale=.7, inner sep=1mm, baseline]
\tikzstyle{rectangular}=[scale=.75, inner sep=1mm, baseline=-.1cm]
\tikzstyle{mid>}=[decoration={markings, mark=at position 0.5 with {\arrow{>}}}, postaction={decorate}]
\tikzstyle{mid<}=[decoration={markings, mark=at position 0.5 with {\arrow{<}}}, postaction={decorate}]
\tikzstyle{over}=[double, draw=white, super thick, double=]
\tikzstyle{snake}=[decorate, decoration={snake, segment length=1mm, amplitude=.3mm}]
\tikzstyle{saw}=[decorate, decoration={saw, segment length=.7mm, amplitude=.25mm}]

\tikzset{super thick/.style={line width=3pt}}
\tikzstyle{knot}=[preaction={super thick, white, draw}]

\tikzstyle{coupon}=[draw, very thick, rectangle, rounded corners=5pt]
\tikzset{Rightarrow/.style={double equal sign distance,>={Implies},->},
triplecd/.style={-,preaction={draw,Rightarrow}},
quadruplecd/.style={preaction={draw,Rightarrow,
shorten >=0pt
},
shorten >=1pt,
-,double,double
distance=0.2pt}}
\tikzset{
    tripleline/.style args={[#1] in [#2] in [#3]}{
        #1,preaction={preaction={draw,#3},draw,#2}
    }
}
\tikzstyle{triple}=[tripleline={[line width=.15mm,black] in
      [line width=.7mm,white] in
      [line width=1mm,black]}] 
\tikzset{
    quadrupleline/.style args={[#1] in [#2] in [#3] in [#4]}{
        #1,preaction={preaction={preaction={draw,#4},draw,#3}, draw,#2}
    }
}
\tikzstyle{quadruple}=[quadrupleline={[line width=.3mm,white] in
      [line width=.6mm,black] in
      [line width=1.2mm,white] in
      [line width=1.5mm,black]}]

\newcommand{\roundNbox}[6]{
	\draw[rounded corners=5pt, very thick, #1] ($#2+(-#3,-#3)+(-#4,0)$) rectangle ($#2+(#3,#3)+(#5,0)$);
	\coordinate (ZZa) at ($#2+(-#4,0)$);
	\coordinate (ZZb) at ($#2+(#5,0)$);
	\node at ($1/2*(ZZa)+1/2*(ZZb)$) {#6};
}

\newcommand{\tikzmath}[2][]
     {\vcenter{\hbox{\begin{tikzpicture}[#1]#2
                     \end{tikzpicture}}}
     }

\makeatletter
\newcommand{\xMapsto}[2][]{\ext@arrow 0599{\Mapstofill@}{#1}{#2}}
\def\Mapstofill@{\arrowfill@{\Mapstochar\Relbar}\Relbar\Rightarrow}
\makeatother

\theoremstyle{plain}
\newtheorem{thm}{Theorem}[section]
\newtheorem*{thm*}{Theorem}
\newtheorem{thmalpha}{Theorem}

\newtheorem{cor}[thm]{Corollary}

\newtheorem*{cor*}{Corollary}

\newtheorem*{conj*}{Conjecture}
\newtheorem{lem}[thm]{Lemma}
\newtheorem*{lem*}{Lemma}

\newtheorem{prop}[thm]{Proposition}

\newtheorem*{quest*}{Question}
\newtheorem*{claim*}{Claim}

\theoremstyle{definition}
\newtheorem{defn}[thm]{Definition}

\newtheorem{sub-ex}[thm]{Sub-Example}
\newtheorem{counter-ex}[thm]{Counter-Example}
\newtheorem{rem}[thm]{Remark}
\newtheorem*{rem*}{Remark}

\definecolor{dark-red}{rgb}{0.7,0.25,0.25}
\definecolor{dark-blue}{rgb}{0.15,0.15,0.55}
\definecolor{medium-blue}{rgb}{0,0,.8}
\definecolor{DarkGreen}{RGB}{0,150,0}
\definecolor{RealPurple}{RGB}{153,0,255}
\definecolor{rho}{named}{red}
\hypersetup{
   colorlinks, linkcolor={purple},
   citecolor={medium-blue}, urlcolor={medium-blue}
}

\newcommand{\id}{\operatorname{id}}

\newcommand{\Irr}{\operatorname{Irr}}

\newcommand{\op}{\operatorname{op}}

\newcommand{\Tr}{\operatorname{Tr}}
\newcommand{\Hom}{\operatorname{Hom}}

\newcommand{\End}{\operatorname{End}}

\DeclareMathOperator{\ev}{ev}
\DeclareMathOperator{\coev}{coev}

\newcommand{\FusCat}{\mathsf{FusCat}}

\def\altdb{\vadjust{\vbox to 0pt{\vss\hbox{\kern \hsize
\quad{\dbend}}\kern\baselineskip\kern-10pt}}}

\newcommand{\noshow}[1]{}
\renewcommand{\MR}[1]{}

\newcommand{\x}{0.3cm }
\newcommand{\h}{6}

\newcommand{\lefttube}[6][]{ %args are front angle, back angle, scaling, height, width, tikz arguments
\ifnum \numexpr 360-#2 > #3
\tikzmath[#1]{
\draw (0,0) -- +(0,#4*#5) arc (180:#2:#4*#6 and #4) -- +(0,-#4*#5) arc(#2:180:#4*#6 and #4);
\draw[dashed] (0,0) arc (180:{360-#2}:#4*#6 and #4);
\draw ({cos(360-#2)*#4*#6+#4*#6},{sin(360-#2)*#4*#6-#4}) arc ({360-#2}:#3:#4*#6 and #4) -- +(0,#4*#5) arc (#3:180:#4*#6 and #4);
}
\else
\tikzmath[#1]{
\draw (0,0) -- +(0,#4*#5) arc (180:#2:#4*#6 and #4) -- +(0,-#4*#5) arc(#2:180:#4*#6 and #4);
\draw[dashed] (0,0) arc (180:#3:#4*#6 and #4) -- +(0,{#4*#5-2*sin(#3)*#4});
\draw ({cos(#3)*#4*#6+#4*#6},{-sin(#3)*#4+#4*#5}) -- +(0,{2*sin(#3)*#4}) arc (#3:180:#4*#6 and #4);
}
\fi
}

\newcommand{\righttube}[6][]{ %args are front angle, back angle, scaling, height, width
\lefttube[xscale=-1,#1]{\numexpr 540-#2}{\numexpr 180-#3}{#4}{#5}{#6}
}

\newcommand{\tubeguts}[5]{
\ifnum \numexpr 360-#1 > #2

\draw (0,0) -- +(0,#3*#4) arc (180:#1:#3*#5 and #3) -- +(0,-#3*#4) arc(#1:180:#3*#5 and #3);
\draw[dashed] (0,0) arc (180:{360-#1}:#3*#5 and #3);
\draw ({cos(360-#1)*#3*#5+#3*#5},{sin(360-#1)*#3*#5-#3}) arc ({360-#1}:#2:#3*#5 and #3) -- +(0,#3*#4) arc (#2:180:#3*#5 and #3);

\else

\draw (0,0) -- +(0,#3*#4) arc (180:#1:#3*#5 and #3) -- +(0,-#3*#4) arc(#1:180:#3*#5 and #3);
\draw[dashed] (0,0) arc (180:#2:#3*#5 and #3) -- +(0,{#3*#4-2*sin(#2)*#3});
\draw ({cos(#2)*#3*#5+#3*#5},{-sin(#2)*#3+#3*#4}) -- +(0,{2*sin(#2)*#3}) arc (#2:180:#3*#5 and #3);

\fi
}

\newcommand{\saddleguts}[6]{
\draw ({#3*#6*cos(#1)},{#3*sin(#1)}) arc (#1:360:{#3*#6} and #3);
\draw[dashed] ({#3*#6},0) arc (0:#2:{#3*#6} and #3) -- ({#3*#6*cos(#2)},{#3*(#4-sin(#2))});
\draw ({#3*#6*cos(#2)},{#3*(#4-sin(#2))}) -- ({#3*#6*cos(#2)},{#3*#4+#3*sin(#2)}) -- +({#3*#5},0) -- +({#3*#5},{-#3*#4}) arc (#2:{360-#1}:{#3*#6} and #3); %smth wrong here
\draw[dashed] ({#3*#5},{#3}) arc (90:180:{#3*#6} and #3);
\draw ({#3*#5-#3*#6},0) arc (180:#1:{#3*#6} and #3) -- +(0,{#3*#4}) -- +(-{#3*#5},{#3*#4}) -- +(-{#3*#5},0);
\draw ({#3*#5-#3*#6},0) arc (0:180:{#3*(#5-2*#6)*.5});
}
\newcommand{\undersaddleguts}[6]{
\draw (0,0) -- +(0,#3*#4) arc (#1:#2:#3*#6 and #3) -- +(0,-#3*#4) -- (0,{(sin(#2)-sin(#1))*#3});
\draw[dashed] (0,{(sin(#2)-sin(#1))*#3}) -- +({((cos(#2)-cos(#1))*#6+#5)*#3},0) -- +({((cos(#2)-cos(#1))*#6+#5)*#3},{#3*(#4-sin(#2)+sin(#1))});
\draw ({((cos(#2)-cos(#1))*#6+#5)*#3}, {#3*#4}) -- %i think this is wrong but w/e
({((cos(#2)-cos(#1))*#6+#5)*#3},{(sin(#2)-sin(#1))*#3+#3*#4}) arc (#2:{360+#1}:#3*#6 and #3) -- +(0,-#3*#4) -- (0,0);
\draw ({(1-cos(#1))*#3*#6},{(-sin(#1))*#3+#3*#4}) arc (-180:0:{#3*(#5-4)*.5});
}

\newcommand{\saddle}[7][]{
\ifnum \numexpr 1 > 0%360-#2 > #3
\tikzmath[#1]{
\saddleguts{#2}{#3}{#4}{#5}{#6}{#7};
}
\else
\tikzmath[xscale=-1, #1]{ %can't figure out how to make this work so we're just gonna stick with left saddles for now. right saddles probably not necessary anyways.
\saddleguts{540-#2}{180-#3}{#4}{#5}{#6}{#7};
}
\fi
}

\title{Anchored planar algebras and 3-categorical graphical calculus}
\author{Brett Hungar}
\date{\today}

\begin{document}

\maketitle

%\tableofcontents

\begin{abstract}
Anchored planar algebras, a generalized notion of Vaughan Jones' planar algebras, have recently seen use in higher category theory, functional analysis, and TQFT applications. These algebras are equipped with a natural 3-dimensional graphical calculus. We compare this graphical calculus with the 3-dimensional graphical calculus associated to tricategories, and we show that anchored planar algebras can be thought of as living in a particular tricategory. This allows for more general techniques to be applied to anchored planar algebras and expands the types of diagrams that anchored planar algebras can interpret.
\end{abstract}

\section{Introduction}

%There has been work recently on anchored planar algebras \nn{}
Anchored planar algebras were first introduced in \cite{HPT2016}, building off Vaughan Jones' notion of planar algebras \cite{jones}. These algebras have already seen use in higher category theory and functional analysis as in \cite{HPTclassification}, and related graphical calculi have applications in physics \cite{physics,physics2}. They are also related to notions of categorified trace, such as the categorified traces discussed in \cite{HPT2015} and \cite{ponto}. Anchored planar algebras, as algebras over an operad, have a natural graphical calculus arising from that operad. This graphical calculus can be seen to be 3-dimensional by viewing anchored planar tangles as strings on tubes that are allowed to merge and twist. The prototypical example of 3-dimensional graphical calculus comes from 3-categories themselves, viewing objects as regions, 1-cells as sheets, 2-cells as strings, and 3-cells as junctions. It is reasonable to wonder whether there is a connection between these two calculi. In this work, we show that we can interpret the operadic graphical calculus within the graphical calculus of 3-categories.

\begin{thmalpha}\label{thm:3dapa}
    For every planar pivotal 3-category $\cC$, objects $A,B$, morphisms $M:A \to B$ and $M^*:B\to A$, and adjoint pair $F:M\boxtimes_B M^* \to A$, $F^*:A \to M\boxtimes_B M^*$ with isomorphisms
    \[ \tau^\ell:(\id_M\boxtimes -)\circ F \to (-^*\boxtimes \id_M^*)\circ F,\ \ \tau^r:F^*\circ(\id_M\boxtimes -) \to F^*\circ(-^*\boxtimes \id_M^*), \]
    satisfying certain axioms, we can construct an anchored planar algebra internal to $\End(1_A)$ which interprets anchored planar tangles via the graphical calculus of $\cC$.
\end{thmalpha}
%this isn't really a theorem BUT i think it bears mentioning in the introduction? At least to take some of the load-bearing words out of the actual main theorem...

Moreover, all anchored planar algebras in braided fusion categories arise this way.

\begin{thmalpha}\label{thm:main}
    For every anchored planar algebra $\cP$, there is a 3-category $\cC$ and objects as in Theorem \ref{thm:3dapa} such that the anchored planar algebra internal to the graphical calculus of $\cC$ is equivalent to $\cP$. 
\end{thmalpha}

Morally, we should think of this theorem as lending evidence to the idea that all 3D graphical calculi are subexamples of the 3D graphical calculus of a 3-category. In Section \ref{sec:background}, we recall the graphical definition of anchored planar algebras and describe the 3-category in which our construction in Theorem \ref{thm:3dapa} takes place.

Our proof proceeds by taking a pivotal module fusion category for a braided pivotal fusion category and showing that it naturally gives rise to certain maps inside a particular 3-category that define an anchored planar algebra. Moreover, we show that this anchored planar algebra is equivalent to the one given by the construction in \cite{HPT2016}. Since it is already known that every anchored planar algebra arises this way (shown in \cite{HPT2016}), this also shows that every anchored planar algebra can be interpreted internal to the graphical calculus of some 3-category. We discuss this construction and prove Theorem \ref{thm:main} in Section \ref{sec:main}.

Since the graphical calculus of 3-categories can interpret more geometric shapes than the anchored planar operad, this also enriches the algebraic structure of anchored planar algebras. We are able to interpret surfaces of higher genus and comultiplication-like operations, in addition to the generating operations of the anchored planar operad. In Section \ref{sec:apps}, we use this structure to answer a question from \cite[1.1.3]{HPTunitary}.

%%%%%%%%%%%%%%%%%%%%%%%%%%%%%%%%%%%%%%%%%%%%%%%%%%%%%%%%%%%%%%%%

\section{Background}\label{sec:background}

%%%%%%%%%%%%%%%%%%%%%%%%%%%%%%%%%%%%%%%%%%%%%%%%%%%%%%%%%%%%%%%%

%\subsection{Anchored planar algebras}

%DISCUSSION OF ANCHORED PLANAR ALGEBRAS

%%%%%%%%%%%%%%%%%%%%%%%%%%%%%%%%%%%%%%%%%%%%%%%%%%%%%%%%%%%%%%%%

%\subsection{Enriched fusion categories}

%\todo{definition, module cats, etc. what do we mean by irr in this context - etc}

\subsection{3D graphical calculus of $\cA-\FusCat$}

Our general setting is the 3-category $\cA-\FusCat$. We begin with a brief discussion of the structure of this 3-category and the various graphical calculi we will use within $\cA-\FusCat$. Our construction of fusion and braided fusion categories, especially in the enriched setting, is based on \cite{egno}, \cite{enriched}, \cite{BJS}, and \cite{center}.

\begin{defn}[\cite{BJS},\cite{enriched}]
Let $\cA$ be a braided fusion category. The 3-category $\cA-\FusCat$ has objects $\cA$-enriched fusion categories (that is, fusion categories $\cX$ with a braided monoidal functor $\cA \to \cZ(\cX)$); a 1-cell ${}_\cX\cM_\cY$ is a (finitely semisimple) $\cX-\cY$ bimodule category $\cM$ equipped with an $\cA$-centered structure; 2-cells are given by bimodule functors; and 3-cells are bimodule natural transformations. Composition of 1-cells is given by relative Deligne product. By a $\cA$-centered structure on ${}_\cX\cM_\cY$, we mean a natural isomorphism between the left $\cA$-action given by $\cA \to \cZ(\cX)$ and the right action given by $\cA \to \cZ(\cY)$. This isomorphism must satisfy appropriate hexagon identities (with the action associators) and must intertwine the tensorators of the $\cX$ and $\cY$ actions.
\end{defn}

%\nn{stuff}
\begin{lem}[\cite{BJS},\cite{center}]\label{lem:setup}
Given a braided fusion category $\cB$ and a $\cB$-enriched fusion category $\cY$, there exists a braided fusion category $\cA$, object $\cX \in \cA-\FusCat$, and 1-cell ${}_\cA\cM_\cX$ with monoidal (resp. braided) equivalences $\End^\cA(\cM_\cX) = \End_{\cA-\FusCat}(\cM_\cX) \cong \cY$ and  $\cZ^\cA(\cX) = \End_{\cA-\FusCat}(\cX) \cong \cB$.
\end{lem}
\begin{proof}
This is well-known; we briefly recall the construction given in \cite{center}. Choose $\cA = \cB$ and $\cX = \cB$. Certainly we have a braided monoidal embedding $\cB \hookrightarrow \cZ^\cB(\cB)$; we just need to show this is essentially surjective. Given an object $b$ and enriched half-braiding $\theta$ for $b$, notice that enriched naturality means the diagram
\[
\begin{tikzcd}
{\left[x,y\right]_\cB} \arrow[r, "-\otimes\id"] \arrow[d, "\id\otimes-"] & {\left[x\otimes b,y\otimes b\right]_\cB} \arrow[d, "\theta_x"] \\
{\left[b\otimes x,b\otimes y\right]_\cB} \arrow[r, "\theta_y"] & {\left[b\otimes x,y\otimes b\right]}
\end{tikzcd}
\]
commutes for all $x,y \in \cB$. Recall that in a fusion category, internal hom is given by $[x,y]_\cB \cong x\otimes y^*$. Then setting $x = 1_\cB$, we have that the diagram
\[
\begin{tikzcd}
y \arrow[r, "\eta_b"] \arrow[d, "\beta_{by}\circ\eta_b"] & y\otimes b \otimes b^* \arrow[d, "\theta_1 = \id"] \\
b\otimes y \otimes b^* \arrow[r, "\theta_y"] & y\otimes b \otimes b^*
\end{tikzcd}
\]
commutes for all $y$, so $\theta_y = \beta_{b,y}$ (where $\beta$ is the braiding in $\cB$). Thus an enriched half-braiding must be the same as the original braiding, so the map $\cB \to \cZ^\cB(\cB)$ is essentially surjective and an equivalence. For $\cM$, choose an indecomposable $\cB$-module; it follows that $\End(\cM_\cB) \cong \cB$.
\end{proof}

In the arguments that follow, we frequently switch between two different graphical calculi for Gray categories, using the coherence result in \cite{gray_coherence} that every 3-category is triequivalent to a Gray category. The first graphical calculus is the 3D graphical calculus found in \cite{gray_graphical}. In our graphical calculations, the Deligne tensor product should be read back to front, functor composition read left to right, and natural transformation composition bottom to top. The second graphical calculus is written using string diagrams plus arrows representing transformations of string diagrams. In this graphical calculus, Deligne tensor product is read left to right, functor composition bottom to top, and natural transformation composition is in the time direction. Both of these graphical calculi are coherent for general 3-categories by \cite{gray_coherence}, \cite{gray_graphical}, and \cite{guthmann}. We illustrate the two calculi in Figure \ref{fig:ex1}.

\begin{figure}
\[
\tikzmath{
\draw[thick] (0,0) -- ++(0,2) -- ++(2,0) -- ++(0,-2) -- ++(-2,0);
\draw[thick] (1,0) -- ++(0,2);
\draw[thick] (0,.5) -- ++(-.5,0) -- ++(0,2) -- ++(2,0) -- ++(0,-.5);
\filldraw (1,1) circle (.05);
\node at (0.25,.25) {$\scriptstyle \mathcal{C}$};
\node at (1.75,.25) {$\scriptstyle \mathcal{D}$};
\node at (.75,.25) {$\scriptstyle F$};
\node at (.75,1.75) {$\scriptstyle G$};
\node at (.75,1.0) {$\scriptstyle \alpha$};
\node at (-.25,2.25) {$\scriptstyle \mathcal{E}$};
}
\qquad
\qquad
\begin{tikzcd}
\tikzmath{
\draw[thick] (0,0) -- ++(0,1);
\draw[thick] (.5,0) -- ++(0,.5) coordinate (a) -- ++(0,.5);
\filldraw (a) circle (.05);
\node at (0,-.3) {$\scriptstyle \mathcal{E}$};
\node at (.5,-.3) {$\scriptstyle \mathcal{C}$};
\node at (.3,.4) {$\scriptstyle F$};
\node at (.5,1.1) {$\scriptstyle \mathcal{D}$};
}
 \arrow[Rightarrow,r,"\alpha"] &
\tikzmath{
\draw[thick] (0,0) -- ++(0,1);
\draw[thick] (.5,0) -- ++(0,.5) coordinate (a) -- ++(0,.5);
\filldraw (a) circle (.05);
\node at (0,-.3) {$\scriptstyle \mathcal{E}$};
\node at (.5,-.3) {$\scriptstyle \mathcal{C}$};
\node at (.3,.4) {$\scriptstyle G$};
\node at (.5,1.1) {$\scriptstyle \mathcal{D}$};
}
\end{tikzcd}
\]
\caption{A demonstration of the 3D graphical calculus and 2D string diagram graphical calculus. Both images represent $\id_{\id_\cE}\boxtimes \alpha$ as a morphism from $\id_{\cE}\boxtimes F \to \id_{\cE}\boxtimes G$, where $\alpha:F \Rightarrow G$ and $F,G:\cC \to \cD$. }\label{fig:ex1}
\end{figure}
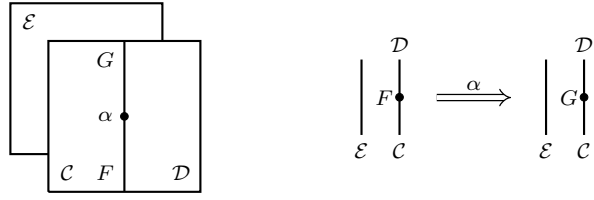

%\todo{include pictures demonstrating these two graphical calculi}

In general, we will use brackets to denote homs, with a subscript denoting the enriching category. When the brackets are undecorated, they mean the underlying hom set as a vector space. Sometimes we will refer to ``elements'' of a hom space $[x,y]$; although this is an object of $\cA$ and not a set, this phrasing is meant to indicate a Yoneda-style generalized elements argument. We spell out the details of any generalized elements argument for enriching categories other than $\cA$, and choose to elide these arguments for $\cA$ for brevity.

%%%%%%%%%%%%%%%%%%%%%%%%%%%%%%%%%%%%%%%%%%%%%%%%%%%%%%%%%%%%%%%%

\section{Anchored planar algebras in $\cA-\FusCat$ MB}\label{sec:main}

Let $\cB$ be an arbitrary braided pivotal tensor category and $\cY$ a pivotal $\cB$-module fusion category. By Lemma \ref{lem:setup}, there exists a braided fusion category $\cA$, an $\cA$-enriched fusion category $\cX$, and an $\cX$-module category $\cM$ such that $\cB \cong \cZ^\cA(\cX)$ and $\cY \cong \End^\cA({}_\cX\cM)$. Moreover, the action of $\cB$ on $\cY$ is equivalent to the action of $\cZ^\cA(\cX) \cong \End^\cA({}_\cX\cX_\cX)$ on $\End({}_\cX\cM)$. Since $\cY \cong \End({}_\cX\cM)$ is a fusion category, this action can be denoted by a monoidal functor $\Phi:\cB \to \cY$ which factors through the Drinfeld center of $\cY$.

With this setup, the graphical calculus of strings on tubes for $\cB$ and $\cY$ as discussed in \cite{HPT2015} can be thought of as occurring in the graphical calculus for 3-categories in $\cA-\FusCat$. To clarify our notation, we will use $\boxtimes$ for the (relative) Deligne product of module categories, $\otimes$ for functor composition, and juxtaposition or $\circ$ for natural transformation composition (written in composition order). Some diagrams use the 3-dimensional graphical calculus in $\cA-\FusCat$, while others project onto a 2-dimensional graphical calculus of string diagrams with morphisms between them.

%\todo{include diagram demonstrating this}

Let $\cM$ be denoted by a sheet. In this context, objects of $\cY \cong \End({}_\cX\cM)$ can be represented by strings on the sheet. The action of $\cY$ on $\cM$ is given by the endofunctor, while the action of $\cY$ on $\cM^\op$ is given by its adjoint.

%%%%%%%%%%%%%%%%%%%%%%%%%%%%%%%%%%%%%%%%%%%%%%%%%%%%%%%%%%%%%%%%

\subsection{Construction of an anchored planar algebra in $\cA-\FusCat$}\label{sec:const}

%%%%%%%%%%%%%%%%%%%%%%%%%%%%%%%%%%%%%%%%%%%%%%%%%%%%%%%%%%%%%%%%

As all the categories involved are finite semisimple, the action functor $-\rhd  m:\cX \to \cM$ has a right adjoint, denoted $[m,-]_\cX:\cM \to \cX$. This assembles into an enriched internal hom $[-,-]_\cX:\cM^\op\boxtimes\cM \to \cX$. On objects, this is given by
\[
[m,n]_\cX := \bigoplus_{x\in\Irr(\cX)} [x\rhd m,n]\otimes x.
\]
Moreover, this has a natural structure as an $\cA$-balanced functor; notice that
\begin{align*}
[m^\op \lhd a,n]_\cX &= [(a^*\rhd m)^\op,n]_\cX \cong \bigoplus_{x\in\Irr(\cX)} [x\rhd (a^*\rhd m),n]\otimes x
\\ &\cong \bigoplus_{x\in\Irr(\cX)} [a^*\rhd(x\rhd m),n]\otimes x \cong \bigoplus_{x\in\Irr(\cX)} [x\rhd m,a\rhd n]\otimes x \cong [m,a\rhd n]_\cX
\end{align*}
Therefore, by the universal property, this induces a functor
\[
\begin{tikzcd}
\cM^\op \boxtimes \cM \arrow[d,hook] \arrow[dr,"{[-,-]_\cX}"] \\
\cM^\op \boxtimes_\cA \cM \arrow[r,dashed] & \cX.
\end{tikzcd}
\]
The induced functor from $\cM^\op \boxtimes_\cA \cM$ to $\cX$ has a left adjoint by semisimplicity. We now describe this left adjoint explicitly. Using the ladder category model for relative Deligne tensor product in \cite{ladder}, we can think of $\cM^\op\boxtimes_\cA\cM$ as the Cauchy completion of the category with objects given by $s^\op\boxtimes t$ for $s,t$ simple in $\cM$, where homs are
\[
\Hom_{\cM^\op \boxtimes_\cA \cM}(s^\op\boxtimes t,u^\op\boxtimes v) = \bigoplus_{a\text{ simple} \in \cA} \Hom_{\cM^\op}(s^\op,u^\op\lhd a) \otimes \Hom_{\cM}(a\rhd t,v);
\]
composition is given by stacking ladders.

Consider the object $\bigoplus_{s\in\Irr(\cM)} s^\op\boxtimes s$ in $\cM^\op\boxtimes \cM$. Choose a basis $\beta^i_{a,s,t}$ for $[a\rhd s,t]$ and corresponding dual basis $(\beta^i_{a,s,t})^*$ for $[t,a\rhd s]$. Then the map
\[\bigoplus_{a,s,t} \sum_i
\tikzmath{
\draw[thick] (0,0) coordinate (m) -- ++(0,.4) coordinate (1) -- ++(.3,.3) coordinate (s);
\draw[thick] (1) -- ++(-.2,.2) arc (45:180:.2) -- ++(0,-.45) coordinate (c);
\node[yshift=-.15cm] at (c) {$\scriptstyle a^*$};
\node[yshift=-.15cm] at (m) {$\scriptstyle t$};
\node[yshift=.1cm,xshift=.1cm] at (s) {$\scriptstyle s$};
\filldraw[fill=red,draw=red,thick] (1) circle (.05);
}
\otimes
\tikzmath{
\draw[thick] (0,0) coordinate (c) -- ++(.2,.4) coordinate (1) -- ++(.2,-.4) coordinate (m);
\draw[thick] (1) -- ++(0,.3) coordinate (s);
\node[yshift=.1cm] at (s) {$\scriptstyle t$};
\node[xshift=-.1cm,yshift=-.1cm] at (c) {$\scriptstyle a$};
\node[xshift=.1cm,yshift=-.1cm] at (m) {$\scriptstyle s$};
\filldraw[fill=red,draw=red,thick] (1) circle (.05);
}
\]
lying in $\bigoplus_{s,t} \bigoplus_a [a^*\rhd t,s]\otimes[a\rhd s,t] \cong \bigoplus_{s,t} \bigoplus_a [s^\op,t^\op\lhd a]\otimes [a\rhd s,t]$ (where the red vertices represent corresponding elements of the dual bases $\beta_{a,s,t}^i$ and $(\beta_{a,s,t}^i)^*$), defines an endomorphism $e$ of $\bigoplus s^\op\boxtimes s$. A routine calculation verifies that $e$ is an idempotent. Denote by $S$ the splitting of the idempotent $e$ in the relative Deligne category $\cM^\op\boxtimes_\cA \cM$.

\begin{rem}
A unitary version of the object $S$ appears in \cite[Remark 4.18]{symmalg}, where it is called the \emph{enriched symmetric enveloping algebra object}. 
\end{rem}

\begin{prop}\label{prop:biadjoint}
The unique $(\cX-\cX)$-bimodule functor $F:\cX \to \cM^\op\boxtimes_\cA\cM$ given by $1_\cX \mapsto S$ is left adjoint to the enriched inner hom functor $\cM^\op\boxtimes_\cA\cM \to \cX$. In fact, these functors are biadjoint.
\end{prop}
\begin{proof}
We begin by describing morphisms $f:\bigoplus s^\op\boxtimes s \to m^\op\boxtimes n$ satisfying $f\circ e = f$ (for arbitrary $m,n \in \cM$), since these correspond to maps out of $S$. %Write
%\[f = \bigoplus_{s,a} \sum_i (f_i:a^*\rhd m \to s)\otimes (g_i:a\rhd s\to n) \in \bigoplus_{s,a} [a^*\rhd m,s]\otimes [a\rhd s,n] \cong \bigoplus_{s,a} [s^\op,m^\op\lhd a]\otimes [a\rhd s,n]\]
Graphically, we denote
\[
f = \bigoplus_{s,a} \sum_i
\tikzmath{
\draw[thick] (0,0) coordinate (a) -- ++(.2,.4) coordinate (1) -- +(0,.4) coordinate (s);
\draw[thick] (1) -- +(.2,-.4) coordinate (m);
%\filldraw[fill=blue,draw=blue,thick] (1) circle (.05);
\node[xshift=-.15cm,yshift=-.1cm] at (a) {$\scriptstyle a^*$};
\node[xshift=0cm,yshift=.1cm] at (s) {$\scriptstyle s$};
\node[xshift=.1cm,yshift=-.1cm] at (m) {$\scriptstyle m$};
\roundNbox{fill=white}{(1)}{.2}{0}{0}{$\scriptstyle f_i$}
}
\otimes
\tikzmath{
\draw[thick] (0,0) coordinate (a) -- ++(.2,.4) coordinate (1) -- +(0,.4) coordinate (s);
\draw[thick] (1) -- +(.2,-.4) coordinate (m);
%\filldraw[fill=red,draw=red,thick] (1) circle (.05);
\roundNbox{fill=white}{(1)}{.2}{0}{0}{$\scriptstyle g_i$}
\node[xshift=-.1cm,yshift=-.1cm] at (a) {$\scriptstyle a$};
\node[xshift=0cm,yshift=.1cm] at (s) {$\scriptstyle n$};
\node[xshift=.1cm,yshift=-.1cm] at (m) {$\scriptstyle s$};
}
\in \bigoplus_{s,a} [a^*\rhd m,s]\otimes [a\rhd s,n] \cong \bigoplus_{s,a} [s^\op,m^\op\lhd a]\otimes [a\rhd s,n]
.
\]
Then the composite of $f$ with $e$ becomes 
\[
f\circ e = 
\bigoplus_{s,c}
\sum_{a,b,\gamma^j_{a,b,c}} \sum_{t,\beta^k_{b,s,t}} \sum_i
\tikzmath{
\draw[thick] (0,0) coordinate (c) -- ++(.2,.4) coordinate (1) -- ++(.3,.3) coordinate (2);
\draw[thick] (.7,0) coordinate (t) -- (2);
\draw[thick] (2) -- ++(0,.6) coordinate (j) -- ++(0,.3) coordinate (n);
\draw[thick] (1) -- ++(0,1) arc (180:90:.1) to[out=0,in=135] (j);
%\filldraw[fill=blue,draw=blue,thick] (2) circle (.05);
\roundNbox{fill=white}{(2)}{.2}{0}{0}{$\scriptstyle f_i$}
\node[xshift=-.15cm,yshift=-.1cm] at (c) {$\scriptstyle c^*$};
\node[xshift=.1cm,yshift=-.1cm] at (t) {$\scriptstyle m$};
\node[xshift=-.15cm,yshift=.35cm] at (1) {$\scriptstyle b^*$};
\node[xshift=.2cm,yshift=-.05cm] at (1) {$\scriptstyle a^*$};
\filldraw[fill=blue,draw=blue,thick] (1) circle (.05);
\filldraw[fill=red,draw=red,thick] (j) circle (.05);
\node[xshift=.1cm,yshift=-.15cm] at (j) {$\scriptstyle t$};
\node[xshift=0cm,yshift=.1cm] at (n) {$\scriptstyle s$};
}
\otimes
\tikzmath{
\draw[thick] (0,0) coordinate (c) -- ++(.2,.4) coordinate (1) -- ++(.3,.3) coordinate (2);
\draw[thick] (.7,0) coordinate (t) -- (2);
\draw[thick] (2) -- ++(0,.6) coordinate (j) -- ++(0,.4) coordinate (n);
\draw[thick] (1) -- ++(0,.2) to[out=90,in=225] (j);
%\filldraw[fill=red,draw=red,thick] (j) circle (.05);
\roundNbox{fill=white}{(j)}{.2}{0}{0}{$\scriptstyle g_i$}
\node[xshift=-.1cm,yshift=-.1cm] at (c) {$\scriptstyle c$};
\node[xshift=.1cm,yshift=-.1cm] at (t) {$\scriptstyle s$};
\node[xshift=-.1cm,yshift=.35cm] at (1) {$\scriptstyle a$};
\node[xshift=.2cm,yshift=0cm] at (1) {$\scriptstyle b$};
\filldraw[fill=blue,draw=blue,thick] (1) circle (.05);
\node[xshift=.1cm,yshift=.15cm] at (2) {$\scriptstyle t$};
\filldraw[fill=red,draw=red,thick] (2) circle (.05);
\node[xshift=0cm,yshift=.1cm] at (n) {$\scriptstyle n$};
},
\]
where the red and blue trivalent vertices correspond to $\beta^k_{b,t,s}$ and to corresponding choices of bases $\gamma_{a,b,c}^j:a\otimes b \to c$ and $(\gamma_{a,b,c}^j)^*:c \to a\otimes b$, respectively. The variables $a,b,c$ range over $\Irr(\cA)$, and $s,t$ range over $\Irr(\cM)$.

We can now use an I=H relation in the multifusion category $\begin{bmatrix} \cX & \cM^\op \\ \cM & \cA  \end{bmatrix}$ (as in \cite[Lemma 2.16]{bicommutant}) on the dual bases $\beta^k_{b,t,s},\gamma^j_{a,b,c}$ to obtain
\[
f\circ e = \bigoplus_{s,c} \sum_{r,\delta_{c,s,r}^j} \sum_{a,t,\epsilon_{a,t,r}^k} \sum_i
\tikzmath{
\draw[thick] (0,0) coordinate (m) -- ++(-.2,.4) coordinate (1) -- ++(-.2,-.2) arc (-45:-225:.2) -- ++(.35,.35) coordinate (2);
\draw[thick] (1) -- (2) -- ++(0,.4) coordinate (3) -- ++(.2,.3) coordinate (s);
\draw[thick] (3) -- ++(-.2,.2) arc (45:180:.2) -- ++(0,-1.25) coordinate (c);
%\draw[fill=blue,draw=blue,thick] (1) circle (.05);
\roundNbox{fill=white}{(1)}{.2}{0}{0}{$\scriptstyle f_i$}
\node[xshift=.1cm,yshift=-.15cm] at (m) {$\scriptstyle m$};
\node[xshift=0cm,yshift=-.15cm] at (c) {$\scriptstyle c^*$};
\node[xshift=-.05cm,yshift=-.35cm] at (1) {$\scriptstyle a^*$};
\node[xshift=.05cm,yshift=.35cm] at (1) {$\scriptstyle t$};
\node[xshift=.15cm,yshift=.2cm] at (2) {$\scriptstyle r$};
\node[xshift=.1cm,yshift=.1cm] at (s) {$\scriptstyle s$};
\filldraw[fill=DarkGreen,draw=DarkGreen,thick] (2) circle (.05);
\filldraw[fill=RealPurple,draw=RealPurple,thick] (3) circle (.05);
}
\otimes
\tikzmath{
\draw[thick] (0,0) coordinate (c) -- ++(.2,.4) coordinate (1) -- ++(0,.4) coordinate (2) arc (225:135:.4) coordinate (3) -- ++(0,.4) coordinate (n);
\draw[thick] (1) -- ++(.2,-.4) coordinate (s);
\draw[thick] (2) arc (-45:45:.4);
%\draw[thick,fill=red,draw=red] (3) circle (.05);
\roundNbox{fill=white}{(3)}{.2}{0}{0}{$\scriptstyle g_i$}
\node[xshift=-.1cm,yshift=-.1cm] at (c) {$\scriptstyle c$};
\node[xshift=.1cm,yshift=-.1cm] at (s) {$\scriptstyle s$};
\node[xshift=.15cm,yshift=.2cm] at (1) {$\scriptstyle r$};
\node[xshift=-.25cm,yshift=.25cm] at (2) {$\scriptstyle a$};
\node[xshift=.25cm,yshift=.25cm] at (2) {$\scriptstyle t$};
\node[xshift=0cm,yshift=.1cm] at (n) {$\scriptstyle n$};
\filldraw[fill=DarkGreen,draw=DarkGreen,thick] (2) circle (.05);
\filldraw[fill=RealPurple,draw=RealPurple,thick] (1) circle (.05);
},
\]
where $r$ ranges over the irreducibles in $\cM$ and $\delta_{c,s,r}^j,\epsilon_{a,t,r}^k$ are bases for the spaces $[c\rhd s,r]$ and $[a\rhd t,r]$, respectively. The green and purple trivalent vertices and represent the $\delta_{c,s,r}^j$ and $\epsilon_{a,t,r}^k$, respectively, and corresponding dual bases. Notice that in the case $m,n$ are simple, this reduces to a constant times
\[
\delta_{m\cong n} \bigoplus_{s,c} \sum_{\delta_{c,s,m}^i}
\tikzmath{
\draw[thick] (0,0) coordinate (m) -- ++(0,.4) coordinate (1) -- ++(.3,.3) coordinate (s);
\draw[thick] (1) -- ++(-.2,.2) arc (45:180:.2) -- ++(0,-.45) coordinate (c);
\node[yshift=-.15cm] at (c) {$\scriptstyle c^*$};
\node[yshift=-.15cm] at (m) {$\scriptstyle m$};
\node[yshift=.1cm,xshift=.1cm] at (s) {$\scriptstyle s$};
\filldraw[fill=DarkGreen,draw=DarkGreen,thick] (1) circle (.05);
}
\otimes
\tikzmath{
\draw[thick] (0,0) coordinate (c) -- ++(.2,.4) coordinate (1) -- ++(.2,-.4) coordinate (m);
\draw[thick] (1) -- ++(0,.3) coordinate (s);
\node[yshift=.1cm] at (s) {$\scriptstyle s$};
\node[xshift=-.1cm,yshift=-.1cm] at (c) {$\scriptstyle c$};
\node[xshift=.1cm,yshift=-.1cm] at (m) {$\scriptstyle m$};
\filldraw[fill=DarkGreen,draw=DarkGreen,thick] (1) circle (.05);
}.
\]

We now use this to describe a map $[1_\cX,[m,n]_\cX] \cong [m,n] \to [S,m^\op\boxtimes n]$. It is enough to describe the map in this case by the universal property of Cauchy completion, since $\cM^\op \boxtimes_\cA \cM$ is the Cauchy completion on the subcategory consisting of objects of the form $m^\op\boxtimes n$. In fact, it is enough to just consider the case where $m,n$ are simple. As discussed above, a map $f \in [\bigoplus s^\op\boxtimes s, m^\op\boxtimes n]$ satisfies $f = f\circ e$, which is true if and only if $f$ is in the form
\[
\lambda\tikzmath{
\draw[thick] (0,0) coordinate (m) -- ++(0,.4) coordinate (1) -- ++(.3,.3) coordinate (s);
\draw[thick] (1) -- ++(-.2,.2) arc (45:180:.2) -- ++(0,-.45) coordinate (c);
\node[yshift=-.15cm] at (c) {$\scriptstyle c^*$};
\node[yshift=-.15cm] at (m) {$\scriptstyle m$};
\node[yshift=.1cm,xshift=.1cm] at (s) {$\scriptstyle s$};
\filldraw[fill=DarkGreen,draw=DarkGreen,thick] (1) circle (.05);
}
\otimes
\tikzmath{
\draw[thick] (0,0) coordinate (c) -- ++(.2,.4) coordinate (1) -- ++(.2,-.4) coordinate (m);
\draw[thick] (1) -- ++(0,.3) coordinate (s);
\node[yshift=.1cm] at (s) {$\scriptstyle s$};
\node[xshift=-.1cm,yshift=-.1cm] at (c) {$\scriptstyle c$};
\node[xshift=.1cm,yshift=-.1cm] at (m) {$\scriptstyle m$};
\filldraw[fill=DarkGreen,draw=DarkGreen,thick] (1) circle (.05);
}
\]
for some scalar $\lambda$. If $m = n$, then we define the map $[m,m] \to [S,m^\op\boxtimes m]$ by sending $\lambda\id_m$ to the map described above. This extends naturally to the rest of the objects in $\cM^\op\boxtimes_{\cA}\cM$. Moreover, it is clear that this map is an isomorphism, since the space of maps $f \in [\bigoplus s^\op\boxtimes s, m^\op\boxtimes n]$ satisfying $f = f\circ e$ is one dimensional if $m = n$ and zero dimensional otherwise. This verifies the adjunction. As left adjoints of linear functors between finite semisimple categories are also right adjoints, these two functors are indeed biadjoint.
\end{proof}

%%%%%%%%%%%%%%%%%%%%%%%%%%%%%%%%%%%%%%%
We denote this functor by a half tube
\[
F = \lefttube{250}{90}{0.3}{9}{2}\,.
\]
Note that the region inside the half tube is labeled with the $\cA$-enriched vacuum $\cA$ and the region outside the half tube is labeled with $\cX$ (the back sheet is $\cM^\op$ and the front sheet is $\cM$).

We denote the adjoint $F^* = [-,-]_{\cX}$ by the half tube
\[
F^* = \righttube{250}{90}{0.3}{9}{2}\,.
\]
The unit and counit of these adjunctions  can be represented with half-spheres and saddles, denoted by
\[ \eta^\ell = \tikzmath{
\draw (0,0) arc (0:360:2*\x and \x);
\draw (0,0) arc (0:-180:2*\x);
}\,,\ \ 
\epsilon^\ell = \tikzmath{
\saddleguts{260}{60}{0.3}{6}{8}{2};
}\,,\ \ 
\eta^r = \tikzmath{
\undersaddleguts{-60}{100}{0.3}{6}{8}{2};
}\,,\ \ 
\epsilon^r = \tikzmath{
\draw (0,0) arc (-180:0:2*\x and \x) arc (0:180:2*\x);
\draw[dashed] (0,0) arc (180:0:2*\x and \x);
}\,,
\]
where $\eta^\ell$ and $\epsilon^\ell$ witness the adjunction $F\dashv F^*$ and $\eta^r$ and $\epsilon^r$ witness $F^*\dashv F$.

It is straightforward to verify that the components of these natural transformations are given by

\begin{equation}\label{eq:eta_ell}
    \eta^\ell_{1_\cX} = 1_\cX \xrightarrow{\bigoplus\id\otimes1_\cX} \bigoplus_s [s,s]\otimes 1_\cX \hookrightarrow \bigoplus_{s,x} [x\rhd s,s] \otimes x \xrightarrow{\sim} \bigoplus_s [s,s]_{\cX} \xrightarrow{e} S,
\end{equation}
\begin{equation}\label{eq:epsilon_ell}
    \epsilon^\ell_{t^\op\boxtimes t'} = [t,t']_\cX\rhd S \hookrightarrow \bigoplus_{x,s} s^\op \boxtimes [x\rhd t,t'] x\rhd s \twoheadrightarrow \bigoplus_x t^\op \boxtimes [x\rhd t,t'] x\rhd t \xrightarrow{\ev} t^\op\boxtimes t' ,
\end{equation}
\begin{equation}\label{eq:eta_r}
    \eta^r_{t^\op\boxtimes t'} = t^\op\boxtimes t' \cong \bigoplus_s [t,s]s^\op\boxtimes [s,t']s %\cong \bigoplus_s s^\op\boxtimes [t,s]\otimes[s,t']s
\xrightarrow{\ev} \bigoplus_s s^\op\boxtimes [t,t']s \hookrightarrow \bigoplus_{x,s}s^\op\boxtimes[t,t']_\cX\rhd s \xrightarrow{[t,t']_\cX\rhd e} [t,t']_\cX \rhd S,
\end{equation}
\begin{equation}\label{eq:epsilon_r}
    \epsilon^r_{1_\cX} = S\hookrightarrow \bigoplus_s [s,s]_\cX \cong \bigoplus_{s,x} [x\rhd s,s]\otimes x \twoheadrightarrow \bigoplus_s [s,s]\otimes 1_\cX \cong \bigoplus_s 1_\cX \xrightarrow{\sum \id} 1_\cX,
\end{equation}
where $t,t'$ are simples in $\cM$ and $x$ is a simple in $\cX$.

Not only are $F$ and $F^*$ biadjoint to each other, but also they witness the duality between $\cM$ and $\cM^\op$.

%\todo{include a definition of what it means to be dual in a Gray monoid}

%\nn{REDO THIS PROPOSITION}
\begin{prop}\label{prop:duality}
Let $G:\cA \to \cM\boxtimes_\cX \cM^\op$, $G^*: \cM\boxtimes_\cX\cM^\op \to \cA$ be as $F,F^*$ (replacing $\cM$ with $\cM^\op$). Then $F$ and $G^*$ witness the duality $\cM^\op \cong \cM^*$ and $G$ and $F^*$ witness the duality $\cM \cong (\cM^\op)^*$.
\end{prop}
\begin{proof}
To show that $F$ and $G^*$ witness $\cM^\op \cong \cM^*$, we must demonstrate the existence of cusp isomorphisms $C:(G^*\boxtimes\id_\cM) \circ (\id_\cM\boxtimes F) \to \id_{\cM}$ and $D:(\id_{\cM^\op}\boxtimes G^*)\circ(F\boxtimes\id_{\cM^\op}) \to \id_{\cM^\op}$ and show they satisfy the swallowtail equations. The other duality is similar and omitted.

By definition of $F$ and $G$, the component $C_m$ must be an isomorphism from $m$ to the subjobject of $\bigoplus_s [s,m]_\cA \rhd s$ given by $(G^*\boxtimes \id_\cM)(\id_m\boxtimes e)$ (where $e$ is the idempotent in the definition of $S$). Explicitly, this map is given by
\begin{align*}
\bigoplus_s [s,m]_\cA \rhd s &\xrightarrow{\bigoplus_{a,s,t} \sum_{i} [\ev_a\rhd s\circ\beta_{a,s,t}^i, m]_\cA \rhd s} \bigoplus_{a,s,t} [a^*\rhd t,m]_\cA \rhd s
\\
& \cong \bigoplus_{a,s,t} [t,m]_\cA \otimes a\rhd s \xrightarrow{\bigoplus_{a,s} \sum_i \beta_{a,s,t}^i} \bigoplus_t [t,m]_\cA \rhd t.
\end{align*}
Notice that the isomorphism $[a^*\rhd t,m]_\cA \cong [t,m]_\cA \otimes a$ is 
\begin{align*}
[a^*\rhd t,m]_\cA &\cong \bigoplus_b [b\rhd a^*\rhd t,m]\otimes b \cong \bigoplus_b [a^*\rhd t, b^*\rhd m] \otimes b \cong \bigoplus_{b,c} [a^*,b^*\otimes c] \otimes [c\rhd t,m]\otimes b  \\
&\cong \bigoplus_{b,c} [b,c\otimes a] \otimes [c\rhd t,m] \otimes b \cong \bigoplus_c [c\rhd t,m] \otimes c\otimes a\cong [t,m]_\cA \otimes a.
\end{align*}
Unpacking this, we can describe the idempotent in terms of generalized elements (i.e., in terms of maps $x \to \bigoplus_s [s,m]_\cA \rhd s = \bigoplus_{s,a} [a\rhd s,m]_\cA \otimes a\rhd s$). This requires describing the action of the idempotent on the vector space $\bigoplus_{a,s} [a\rhd s,m] \otimes [x,a\rhd s]$. This map is given by:

\[
\bigoplus_{a,s} \sum_{i} 
\tikzmath{
\draw[thick] (0,0) coordinate (a) -- ++(.2,.4) coordinate (1) -- +(0,.4) coordinate (s);
\draw[thick] (1) -- +(.2,-.4) coordinate (m);
%\filldraw[fill=blue,draw=blue,thick] (1) circle (.05);
\node[xshift=-.15cm,yshift=-.1cm] at (a) {$\scriptstyle a$};
\node[xshift=0cm,yshift=.1cm] at (s) {$\scriptstyle m$};
\node[xshift=.1cm,yshift=-.1cm] at (m) {$\scriptstyle s$};
\roundNbox{fill=white}{(1)}{.2}{0}{0}{$\scriptstyle f_i$}
}
\otimes
\tikzmath{
\draw[thick] (0,0) coordinate (x) -- ++(0,.4) coordinate (1) -- ++(.2,.4) coordinate (s);
\draw[thick] (1) -- +(-.2,.4) coordinate (a);
\roundNbox{fill=white}{(1)}{.2}{0}{0}{$\scriptstyle g_i$}
\node[xshift=-.1cm,yshift=.1cm] at (a) {$\scriptstyle a$};
\node[xshift=.1cm,yshift=.1cm] at (s) {$\scriptstyle s$};
\node[xshift=0cm,yshift=-.1cm] at (x) {$\scriptstyle x$};
}
\mapsto
\bigoplus_{c,t} \sum_{i,\beta_{a,s,t}^j,\gamma_{a,b,c}^k}
\tikzmath{
\draw (0,0) coordinate (c) -- ++(0,.4) coordinate (1) to[out=45,in=270] ++(.4,.4) arc (180:45:.1) -- ++(.1,-.1) coordinate (2);
\draw (1) to[out=135,in=225] ++(.2,.8) coordinate (3) to[out=-45,in=45] (2);
\draw (3) -- ++(0,.4) coordinate (m);
\draw (2) -- ++(0,-.75) coordinate (t);
\roundNbox{fill=white}{(3)}{.2}{0}{0}{$\scriptstyle f_i$}
\node[xshift=0cm,yshift=-.15cm] at (c) {$\scriptstyle c$};
\node[xshift=0cm,yshift=.1cm] at (m) {$\scriptstyle m$};
\node[xshift=0cm,yshift=-.15cm] at (t) {$\scriptstyle t$};
\node[xshift=0cm,yshift=.3cm] at (1) {$\scriptstyle a$};
\node[xshift=.35cm,yshift=.05cm] at (1) {$\scriptstyle b^*$};
\node[xshift=0cm,yshift=.3cm] at (2) {$\scriptstyle s$};
\filldraw[fill=blue,draw=blue,thick] (1) circle (.05);
\filldraw[fill=red,draw=red,thick] (2) circle (.05);
}
\otimes
\tikzmath{
\draw (0,0) coordinate (x) -- ++(0,.4) coordinate (1) to[out=135,in=270] ++(-.4,.3) coordinate (b) to[out=90,in=225] ++(.2,.3) coordinate (2) -- ++(.1,-.1) coordinate (a) arc (225:315:.2) -- ++(.1,.1) coordinate (3) -- ++(0,.2) coordinate (t);
\draw (2) -- ++(0,.2) coordinate (c);
\draw (1) to[out=45,in=-45] (3);
\roundNbox{fill=white}{(1)}{.2}{0}{0}{$\scriptstyle g_i$}
\node[xshift=0cm,yshift=.1cm] at (c) {$\scriptstyle c$};
\node[xshift=.2cm,yshift=.1cm] at (a) {$\scriptstyle b$};
\node[xshift=-.1cm,yshift=0cm] at (b) {$\scriptstyle a$};
\node[xshift=.4cm,yshift=.3cm] at (1) {$\scriptstyle s$};
\node[xshift=0cm,yshift=.1cm] at (t) {$\scriptstyle t$};
\node[xshift=0cm,yshift=-.1cm] at (x) {$\scriptstyle x$};
\filldraw[fill=blue,draw=blue,thick] (2) circle (.05);
\filldraw[fill=red,draw=red,thick] (3) circle (.05);
},
\]
where the red dots range over $\beta_{a,s,t}^j$ as before, and the blue dots range over $\gamma_{a,b,c}^k$ and $(\gamma_{a,b,c}^k)^*$, a choice of basis and dual basis for $[c,a\otimes b^*]$ and $[a\otimes b^*,c]$.

Using an I=H relation, as in the proof of Proposition \ref{prop:biadjoint}, we can rewrite this as
\[
\bigoplus_{c,t} \sum_{i,\delta_{c,t,r}^j,\epsilon_{r,a,s}^k}
\tikzmath{
\draw (0,0) coordinate (c) -- ++(.2,.2) coordinate (1) -- ++(0,.4) coordinate (2) to[out=135,in=225] ++(0,.6) coordinate (3) -- ++(0,.4) coordinate (m);
\draw (2) to[out=45,in=-45] (3);
\draw (.4,0) coordinate (t) -- (1);
\roundNbox{fill=white}{(3)}{.2}{0}{0}{$\scriptstyle f_i$};
\node[xshift=-.1cm,yshift=-.15cm] at (c) {$\scriptstyle c$};
\node[xshift=0cm,yshift=.1cm] at (m) {$\scriptstyle m$};
\node[xshift=.1cm,yshift=-.15cm] at (t) {$\scriptstyle t$};
\filldraw[fill=DarkGreen,draw=DarkGreen,thick] (1) circle (.05);
\filldraw[fill=RealPurple,draw=RealPurple,thick] (2) circle (.05);
}
\otimes
\tikzmath{
\draw (0,0) coordinate (x) -- ++(0,.4) coordinate (1) to[out=135,in=225] ++(0,.6) coordinate (2) -- ++(0,.4) coordinate (3) -- ++(-.2,.2) coordinate (c);
\draw (1) to[out=45,in=-45] (2);
\draw (3) -- ++(.2,.2) coordinate (t);
\roundNbox{fill=white}{(1)}{.2}{0}{0}{$\scriptstyle g_i$};
\node[xshift=-.1cm,yshift=.15cm] at (c) {$\scriptstyle c$};
\node[xshift=0cm,yshift=-.1cm] at (x) {$\scriptstyle x$};
\node[xshift=.1cm,yshift=.15cm] at (t) {$\scriptstyle t$};
\filldraw[fill=DarkGreen,draw=DarkGreen,thick] (2) circle (.05);
\filldraw[fill=RealPurple,draw=RealPurple,thick] (3) circle (.05);
},
\]
with the green and purple circles representing $\delta_{c,t,r}^j$ and $\epsilon_{r,a,s}^k$, another pair of bases and dual bases.

We now show that $m$ is equivalent to the subobject generated by this idempotent. It is enough to do this in the case where $m$ is simple.
Consider the map $i:m \to \bigoplus_s [s,m]_\cA\rhd s \cong \bigoplus_{a,s} [a\rhd s,m]\otimes a\rhd s$
by $\bigoplus_{a,s}\sum_i \beta_{a,s,m}^i \otimes (\beta_{a,s,m}^i)^*$.
It is immediately apparent that this map and its dual split the idempotent above, exhibiting $m$ as the subobject of this isomorphism. This defines the component of $C$ at the object $m$; the isomorphism $D$ is formally dual.
\end{proof}

\begin{lem}\label{lem:etasCscommute}
The natural map
\[\id_\cM \xRightarrow{\eta^\ell_G} (G^*\circ G)\boxtimes \id_\cM \xRightarrow{\eta_F^r} (F^*\boxtimes \id_\cM) \circ (\id_\cM\boxtimes G) \circ(\id_\cM\boxtimes G^*)\circ(F\boxtimes \id_\cM) \xRightarrow{C\circ C} \id_\cM\]
is the identity on $\id_\cM$.
\end{lem}
\begin{proof}
In the graphical calculus of string diagrams, this map is given by
\[
\tikzmath{
\draw (0,0) -- (0,1);
} \implies
\tikzmath{
\draw (0,0) -- (0,1);
\draw (-.25,0.5) arc (0:360:.25);
} \implies
\tikzmath{
\draw (0,0) -- +(0,.25) arc (0:180:.25) arc (0:-180:.25) -- +(0,.75) arc (180:0:.25) arc (-180:0:.25) -- +(0,.25);
} \implies
\tikzmath{
\draw (0,0) -- +(0,1);
}\,.
\]
The following pasting diagram commutes, as it is composed of naturality squares and snake equation bigons:
\begin{center}
\begin{tikzcd}
&
\tikzmath{
\draw (0,0) -- (0,1);
\draw (-.25,0.5) arc (0:360:.25);
}
\arrow[r, Rightarrow, "\eta^r"]
\arrow[d, Rightarrow, "D^{-1}"]
&
\tikzmath{
\draw (0,0) -- +(0,.25) arc (0:180:.25) arc (0:-180:.25) -- +(0,.75) arc (180:0:.25) arc (-180:0:.25) -- +(0,.25);
\draw[dashed, rounded corners = 5pt,red] (-1.125,.63) rectangle ++(1.25,.675);
\draw[dashed, rounded corners = 5pt,blue] (-1.125,-.13) rectangle ++(1.25,.675);
}
\arrow[r,Rightarrow,"\textcolor{blue}{D}\textcolor{red}{C}"]
\arrow[d,Rightarrow,"D^{-1}"]
&
\tikzmath{
\draw (0,0) -- +(0,1);
}
\arrow[dd,Rightarrow,"D^{-1}"]
\\
\tikzmath{
\draw (0,0) -- +(0,1);
}
\arrow[ur,Rightarrow,"\eta^\ell"]
\arrow[dr,Rightarrow,"D^{-1}"']
&
\tikzmath{
\draw (0,0) -- +(0,.25) arc (180:0:.25) arc (-180:0:.25) -- +(0,1);
\draw (0,1) arc (180:0:.25) arc (0:-180:.25);
}
\arrow[r,Rightarrow,"\eta^r"]
\arrow[d,Rightarrow,bend left=10,"\epsilon^\ell"]
&
\tikzmath{
\draw (0,0) -- +(0,.25) arc (180:0:.25) arc (-180:0:.25) -- +(0,.75) arc (0:180:.25) arc (0:-180:.25) -- +(0,.75) arc (180:0:.25) arc (-180:0:.25) -- +(0,.25);
\draw[dashed, rounded corners = 5pt,red] (-.125,1.43) rectangle ++(1.25,.675);
\draw[dashed, rounded corners = 5pt,blue] (-.125,.63) rectangle ++(1.25,.675);
}
\arrow[d,Rightarrow,"\epsilon^\ell"]
\arrow[dr,Rightarrow,"\textcolor{blue}{D}\textcolor{red}{C}"]
\\
& \tikzmath{
\draw (0,0) -- +(0,.25) arc (180:0:.25) arc (-180:0:.25) -- +(0,.25);
}
\arrow[u,Rightarrow,bend left=10,"\eta^\ell"]
\arrow[r,Rightarrow,"\eta^r"]
\arrow[rr,bend right=30,equal]
&
\tikzmath[scale=-1]{
\draw (0,0) -- +(0,.25) arc (180:0:.25) arc (-180:0:.25) -- +(0,1);
\draw (0,1) arc (180:0:.25) arc (0:-180:.25);
}
\arrow[r,Rightarrow,"\epsilon^r"]
&
\tikzmath{
\draw (0,0) -- +(0,.25) arc (180:0:.25) arc (-180:0:.25) -- +(0,.25);
}
\end{tikzcd}.
\end{center}
The result follows.
\end{proof}

By duality, there are similar versions of Lemma \ref{lem:etasCscommute} for $\cM^\op$ and for the composite $(C\circ C)\eta_G^r\eta_F^\ell$.

We denote the map $y \mapsto F^*\circ(\cM^\op\boxtimes y)\circ F$ by $\Tr:\End({}_\cX\cM) \to \cZ^\cA(\cX)$; this corresponds to the categorified trace of \cite{HPT2015}.

%\nn{tragically, redo this as well}
\begin{prop}\label{prop:trace_rightadj}
    The map $\Tr:\End({}_\cX\cM) \to \cZ^\cA(\cX)$ given by $F^* \circ (\cM^\op\otimes -) \circ F$ is right adjoint to the action map $\Phi:\cZ^\cA(\cX) \to \End({}_\cX\cM)$.
\end{prop}
\begin{proof}
    It is enough to show that $[1_\cX,F^*yF] \cong [\id_\cM,y]$ for all $y:\cM \to \cM$ since all the functors involved are module functors. By Proposition \ref{prop:biadjoint}, we have 
\[
[1_\cX,F^*yF] \cong [S,yF] = [S,(\cM^\op\boxtimes y)(S)],
\]
which is isomorphic to the image of $[S,(\cM^\op\boxtimes y)(\bigoplus s^\op\boxtimes s)]$ after cutting down by the idempotent $(\cM^\op\boxtimes y)(e)$. Now $(\cM^\op\boxtimes y)(\bigoplus s^\op\boxtimes s) = \bigoplus s^\op\boxtimes ys \cong \bigoplus_{s,t} s^\op\boxtimes \bbC^{n_{y,s}^t}\otimes t$, where $n_{y,s}^t$ denotes the dimension of $[t,ys]$. Let $n = \max(n_{y,s}^t)$; then this object includes into $\left(\bigoplus s^\op\boxtimes t\right)^{\oplus n}$. Recalling the definition of $e$, it follows that the image of the idempotent $(\cM^\op\boxtimes y)(e)$ is the same as the image of the map
\[ \bigoplus s^\op\boxtimes ys \hookrightarrow \left(\bigoplus s^\op\boxtimes t\right)^{\oplus n} \twoheadrightarrow \left(\bigoplus s^\op\boxtimes s\right)^{\oplus n} \twoheadrightarrow S^{\oplus n}. \]

If we are mapping out of $S$, though, we are already mapping into this copy of $S^{\oplus n}$. Thus $[S,(\cM^\op\boxtimes y)(S)] \cong [S,(\cM^\op\boxtimes y)(\bigoplus s^\op\boxtimes s)]$. Finally, notice that we have $[S,\bigoplus s^\op\boxtimes ys] \cong \bigoplus [s,ys] \cong [\id_\cM,y]$ by Proposition \ref{prop:biadjoint} and semisimplicity. \qedhere

\iffalse 
By definition, we have
    \[F^*yF \cong \bigoplus_{\substack{s\in\cM\text{ simple},\\ x \in \cX\text{ simple}}} [x\rhd s,ys]\otimes x,\]
    and so $[1_\cX,F^*yF] \cong \bigoplus_{s}[s,ys]$. Since $\cM$ is semisimple, this is precisely $[\id_\cM,y]$, so we have the adjunction.\fi
\end{proof}

The maps $\Tr$ and $\Phi$ are represented graphically by
\[ \Tr\left(\,
\tikzmath{
\draw[thick] (0,0) rectangle (1,1);
\draw[red,thick] (.5,0) -- (.5,1);
}
\,\right) =
\tikzmath{
\draw[thick] (0,0) arc (-90:-180:.5 and .25) arc (-180:0:.5 and .25) -- +(0,1) arc (0:540:.5 and .25) -- +(0,-1);
\draw[dashed,thick] (.5,.25) arc (0:180:.5 and .25);
\draw[red,thick] (0,0) -- +(0,1);
}\,,\ \ \ \ 
\Phi\left(\,
\tikzmath{
\draw[DarkGreen,thick] (0,0) -- (0,1);
}
\,\right)
= \tikzmath{
\draw[thick] (0,0) rectangle (1,1);
\draw[DarkGreen,thick,knot] (.6,-.25) -- +(0,1);
}\,.
\]
Moreover, we can construct the unit and the counit of the adjunction $\Phi \dashv \Tr$ with the cups and caps from the graphical calculus above.
\begin{lem}\label{lem:phitradj}
The unit of the adjunction $\Phi\dashv \Tr$ is given by tensoring with $\eta^\ell$, while the counit is given by $(C\circ C)\eta^r$; graphically, the counit and unit are
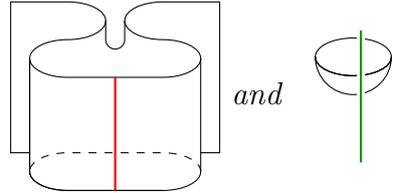
\[
\tikzmath{
\draw (0,0) -- +(-1.25,0) arc (-90:-180:.5 and .25) -- +(0,1.5) arc (-180:-90:.5 and .25) -- +(1.25,0) arc (-90:90:.5 and .25) arc (270:90:.5 and .25) -- +(.75,0) -- +(.75,-2) -- +(.5,-2);
\draw (0,0) -- +(-1.25,0) arc (-90:-180:.5 and .25) -- +(0,1.5) arc (180:90:.5 and .25) arc (-90:90:.5 and .25) -- +(-.75,0) -- +(-.75,-2) -- +(-.5,-2);
\draw (0,0)  arc (-90:0:.5 and .25) -- +(0,1.5);
\draw[dashed] (.5,.25) arc (0:90:.5 and .25) -- +(-1.25,0) arc (90:180:.5 and .25);
\draw (-.5,2.25) -- +(0,-.25) arc (0:-180:.125) -- +(0,.25);
\draw[red,thick] (-.625,0) -- +(0,1.5);
}\text{   and }\ \
\tikzmath{
\draw (-.25,1) arc (-90:-180:.5) arc (-180:360:.5 and .25) arc (0:-90:.5);
\draw[DarkGreen,thick,knot] (-.15,.1) -- +(0,1.75);
}
\,,
\]
respectively (with cusps omitted).
\end{lem}
\begin{proof}
It is enough to show that these maps satisfy the snake equations. In the graphical calculus of string diagrams, one snake is given by
\[
\tikzmath{
\draw (0,0) -- (0,1);
\filldraw[DarkGreen] (.25,.5) circle (.05);
}
\ \xRightarrow{\eta^\ell_F}\ 
\tikzmath{
\draw (0,0) -- (0,1);
\filldraw[DarkGreen] (.5,.75) circle (.05);
\draw (.25,.5) circle (.15); 
}
\ \xRightarrow{\eta^r_G}\ 
\tikzmath{
\draw (0,0) -- (0,.25) arc (180:0:.25) arc (-180:0:.25) -- +(0,0.75) arc (0:180:.25) arc (0:-180:.25) -- +(0,0.5);
\filldraw[DarkGreen] (1.25,1.4) circle (.05);
}
\ \xRightarrow{C\circ C}\ 
\tikzmath{
\draw (0,0) -- (0,1);
\filldraw[DarkGreen] (.25,.5) circle (.05);
},
\]
and the other is given by
\[
\tikzmath{
\draw (0,0) arc (0:360:.25);
\filldraw[red] (0,0) circle (.05);
}
\ \xRightarrow{\eta^\ell_F}
\tikzmath{
\draw (0,0) arc (0:360:.25);
\draw (-.75,0) arc (0:360:.25);
\filldraw[red] (0,0) circle (.05);
}
\ \xRightarrow{\eta_G^r}\ 
\tikzmath{
\draw (0,0) arc (0:-180:.25) arc (0:180:.25) arc (0:-180:.25) -- +(0,.75) arc (180:0:.25) arc (-180:0:.25) arc (180:0:.25) -- (0,0);
\filldraw[red] (0,.375) circle (.05);
}
\ \xRightarrow{C\circ C}\ 
\tikzmath{
\draw (0,0) arc (0:360:.25);
\filldraw[red] (0,0) circle (.05);
}.
\]
Both of these are equal to the identity by an application of Lemma \ref{lem:etasCscommute}.
\end{proof}

Notice that $\epsilon^\ell$ gives rise to a multiplication map $\mu_{x,y}:\Tr(x)\circ\Tr(y) \to \Tr(x\circ y)$, denoted by the pair of pants
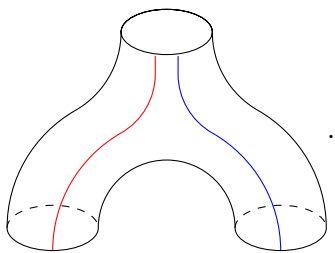
\[
\tikzmath{
\draw (0,0) arc (-180:0:2*\x and \x) arc (180:0:3*\x) arc (-180:0:2*\x and \x) arc (0:60:6*\x) arc (240:180:{4*\x})
arc (0:540:2*\x and \x) %this is the wrap around up top
arc (0:-60:4*\x) arc (120:180:6*\x);
\draw[dashed] (0,0) arc (180:0:2*\x and \x);
\draw[dashed] (10*\x,0) arc (180:0:2*\x and \x);
\draw[red] (2*\x,-\x) arc (180:120:6*\x) arc (300:360:{3*\x}) -- +(0,0.24);
\draw[blue] (12*\x,-\x) arc (0:60:6*\x) arc (240:180:{3*\x}) -- +(0,0.24);
}
.
\]
%It is an easy calculation that the component of the multiplication map at $1_\cX$ is given by
%\[ (\mu_{x,y})_{1_\cX} = F^*xFF^*yF = \bigoplus [s,xs]_\cX \otimes \bigoplus [t,yt]_\cX \cong \bigoplus \left[s,\bigoplus [t,yt]_\cX\otimes xs\right] = \nn{} \]

This multiplication agrees with the multiplication map given in \cite{HPT2015}.
\begin{prop}\label{prop:mult}
    The map $\mu_{x,y}$ is the mate of the map
    \[ \Phi(\Tr(x)\circ\Tr(y)) \xrightarrow{\mu_\Phi}\Phi(\Tr(x)) \circ \Phi(\Tr(y)) \xrightarrow{\epsilon\circ\epsilon} x\circ y \]
    under the adjunction $\Phi\dashv \Tr$.
\end{prop}
\begin{proof}
We want to show that $\mu_{x,y}$ is equal to the map
\[ \Tr(x) \circ \Tr(y) \Rightarrow \Tr(\Phi(\Tr(x)\circ\Tr(y))) \cong \Tr(\Phi(\Tr(x))\circ\Phi(\Tr(y))) \Rightarrow \Tr(x\circ y). \]
From the characterization of the unit and counit of $\Phi \dashv \Tr$ in Lemma \ref{lem:phitradj}, we can write this map in the graphical calculus of string diagrams as
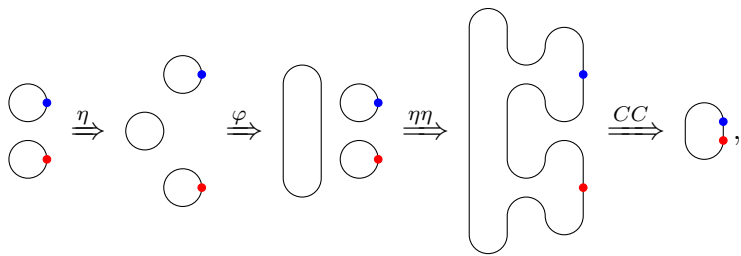
\[
\tikzmath{
\draw (0,0) arc (0:360:.25);
\draw (0,.75) arc (0:360:.25);
\filldraw[red] (0,0) circle (.05);
\filldraw[blue] (0,.75) circle (.05);
}
\ \xRightarrow{\eta}\ 
\tikzmath{
\draw (0,0) arc (0:360:.25);
\draw (0,1.5) arc (0:360:.25);
\draw (-.5,.75) arc (0:360:.25);
\filldraw[red] (0,0) circle (.05);
\filldraw[blue] (0,1.5) circle (.05);
}
\ \xRightarrow{\varphi}\ 
\tikzmath{
\draw (0,0) arc (0:360:.25);
\draw (0,.75) arc (0:360:.25);
\draw (-.75,-.25) arc (0:-180:.25) -- +(0,1.25) arc (180:0:.25) -- (-.75,-.25);
\filldraw[red] (0,0) circle (.05);
\filldraw[blue] (0,.75) circle (.05);
}
\ \xRightarrow{\eta\eta}\ 
\tikzmath{
\draw (0,0) arc (0:-180:.25) arc (0:180:.25) -- +(0,-.25) arc (0:-180:.25) -- +(0,2.75) arc (180:0:.25) -- +(0,-.25) arc (-180:0:.25) arc (180:0:.25) -- +(0,-.75) arc (0:-180:.25) arc (0:180:.25) -- +(0,-.75) arc (-180:0:.25) arc (180:0:.25) -- (0,0);
\filldraw[red] (0,.375) circle (.05);
\filldraw[blue] (0,1.875) circle (.05);
}
\ \xRightarrow{CC}\ 
\tikzmath{
\draw (0,0) arc (0:-180:.25) -- +(0,.25) arc (180:0:.25) -- (0,0);
\filldraw[red] (0,0) circle (.05);
\filldraw[blue] (0,.25) circle (.05);
}\,,
\]
where $\varphi$ is the interchanger. To show this is equal to the multiplication map $\mu_{x,y}$, consider the pasting diagram below:
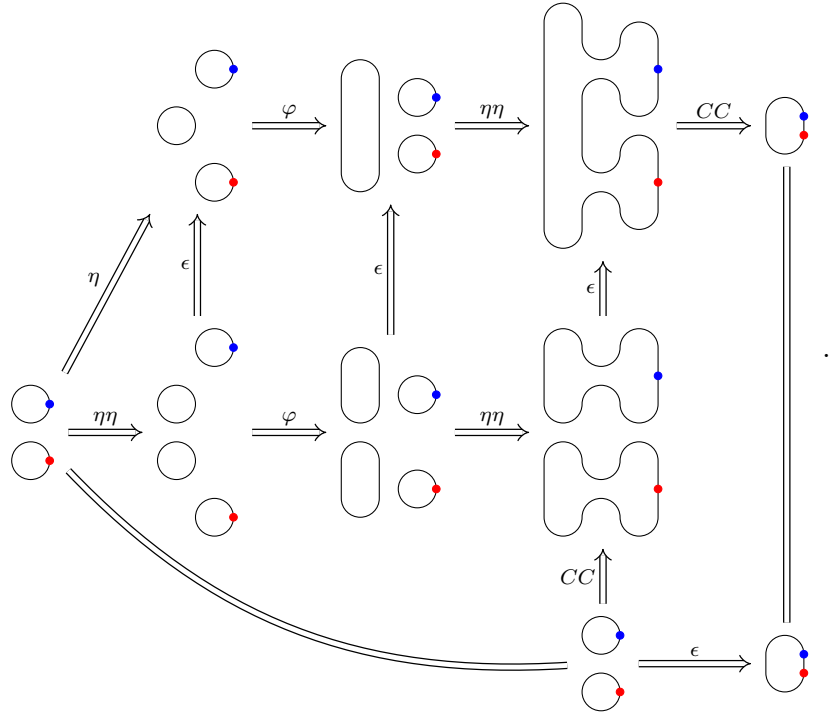
\[
\begin{tikzcd}
&
\tikzmath{
\draw (0,0) arc (0:360:.25);
\draw (0,1.5) arc (0:360:.25);
\draw (-.5,.75) arc (0:360:.25);
\filldraw[red] (0,0) circle (.05);
\filldraw[blue] (0,1.5) circle (.05);
}
\arrow[r, Rightarrow, "\varphi"]
&
\tikzmath{
\draw (0,0) arc (0:360:.25);
\draw (0,.75) arc (0:360:.25);
\draw (-.75,-.25) arc (0:-180:.25) -- +(0,1.25) arc (180:0:.25) -- (-.75,-.25);
\filldraw[red] (0,0) circle (.05);
\filldraw[blue] (0,.75) circle (.05);
}
\arrow[r, Rightarrow, "\eta\eta"]
&
\tikzmath{
\draw (0,0) arc (0:-180:.25) arc (0:180:.25) -- +(0,-.25) arc (0:-180:.25) -- +(0,2.75) arc (180:0:.25) -- +(0,-.25) arc (-180:0:.25) arc (180:0:.25) -- +(0,-.75) arc (0:-180:.25) arc (0:180:.25) -- +(0,-.75) arc (-180:0:.25) arc (180:0:.25) -- (0,0);
\filldraw[red] (0,.375) circle (.05);
\filldraw[blue] (0,1.875) circle (.05);
}
\arrow[r, Rightarrow, "CC"]
&
\tikzmath{
\draw (0,0) arc (0:-180:.25) -- +(0,.25) arc (180:0:.25) -- (0,0);
\filldraw[red] (0,0) circle (.05);
\filldraw[blue] (0,.25) circle (.05);
}
\arrow[dd, equal]
\\
\tikzmath{
\draw (0,0) arc (0:360:.25);
\draw (0,.75) arc (0:360:.25);
\filldraw[red] (0,0) circle (.05);
\filldraw[blue] (0,.75) circle (.05);
}
\arrow[ur, Rightarrow, "\eta"]
\arrow[r, Rightarrow, "\eta\eta"]
\arrow[drrr, equal, bend right = 25]
&
\tikzmath{
\draw (0,0) arc (0:360:.25);
\draw (0,2.25) arc (0:360:.25);
\draw (-.5,.75) arc (0:360:.25);
\draw (-.5,1.5) arc (0:360:.25);
\filldraw[red] (0,0) circle (.05);
\filldraw[blue] (0,2.25) circle (.05);
}
\arrow[u, Rightarrow, "\epsilon"]
\arrow[r, Rightarrow, "\varphi"]
&
\tikzmath{
\draw (0,0) arc (0:360:.25);
\draw (-.75,-.125) arc (0:-180:.25) -- +(0,.5) arc (180:0:.25) -- +(0,-.5);
\draw (0,1.25) arc (0:360:.25);
\draw (-.75,1.125) arc (0:-180:.25) -- +(0,.5) arc (180:0:.25) -- +(0,-.5);
\filldraw[red] (0,0) circle (.05);
\filldraw[blue] (0,1.25) circle (.05);
}
\arrow[u, Rightarrow, "\epsilon"]
\arrow[r, Rightarrow, "\eta\eta"]
&
\tikzmath{
\draw (0,0) arc (0:-180:.25) arc (0:180:.25) arc (0:-180:.25) -- +(0,.75) arc (180:0:.25) arc (-180:0:.25) arc (180:0:.25) -- (0,0);
\filldraw[red] (0,.375) circle (.05);
\draw (0,1.5) arc (0:-180:.25) arc (0:180:.25) arc (0:-180:.25) -- +(0,.75) arc (180:0:.25) arc (-180:0:.25) arc (180:0:.25) -- (0,1.5);
\filldraw[blue] (0,1.875) circle (.05);
}
\arrow[u, Rightarrow, "\epsilon"]
\\
& & &
\tikzmath{
\draw (0,0) arc (0:360:.25);
\draw (0,.75) arc (0:360:.25);
\filldraw[red] (0,0) circle (.05);
\filldraw[blue] (0,.75) circle (.05);
}
\arrow[u, Rightarrow, "CC"]
\arrow[r, Rightarrow, "\epsilon"]
&
\tikzmath{
\draw (0,0) arc (0:-180:.25) -- +(0,.25) arc (180:0:.25) -- (0,0);
\filldraw[red] (0,0) circle (.05);
\filldraw[blue] (0,.25) circle (.05);
}
\end{tikzcd}.
\]
Most of the cells are naturality squares. The upper left triangle commutes by the snake for $\eta^\ell_G,\epsilon^\ell_G$, and the lower left triangle commutes by an application of Lemma \ref{lem:etasCscommute}. The upper map is the mate of $\Phi(\Tr(x)\circ\Tr(y)) \xrightarrow{\mu_\Phi}\Phi(\Tr(x)) \circ \Phi(\Tr(y)) \xrightarrow{\epsilon\circ\epsilon} x\circ y$, as described above, while the lower map is $\mu_{x,y}$. Thus the two multiplications are equal.
\end{proof}

\begin{rem}\label{rem:thirdmult}
There is a third graphical multiplication $\Tr(x)\boxtimes\Tr(y) \to \Tr(x\circ y)$, given by
\[ 
\tikzmath{
\draw (0,0) arc (0:360:.25);
\draw (-.75,0) arc (0:360:.25);
\filldraw[red] (0,0) circle (.05);
\filldraw[blue] (-.75,0) circle (.05);
}
\ \xRightarrow{\eta}\ 
\tikzmath{
\draw (0,0) arc (0:-180:.25) arc (0:180:.25) arc (0:-180:.25) -- +(0,.75) arc (180:0:.25) arc (-180:0:.25) arc (180:0:.25) -- (0,0);
\filldraw[red] (0,.375) circle (.05);
\filldraw[blue] (-1,.75) circle (.05);
}
\ \xRightarrow{CC}\ 
\tikzmath{
\draw (0,0) arc (0:-180:.25) -- +(0,.25) arc (180:0:.25) -- (0,0);
\filldraw[red] (0,0) circle (.05);
\filldraw[blue] (0,.25) circle (.05);
}.
\]
This multiplication also agrees with the two above (under the interchanger $\phi:\Tr(x)\boxtimes\Tr(y) \cong \Tr(x)\circ\Tr(y)$), via a similar proof to above. Lemma \ref{lem:etasCscommute} records the fact that this multiplication similarly absorbs the unit $\eta^\ell$. 
\end{rem}

Since $\Phi \dashv \Tr$, we have canonical traciators, as shown in \cite{HPT2015}. In this context, these traciators can be broken down as the composition of two half twists, denoted graphically by
\[
\tau^\ell_y = \tikzmath{
\tubeguts{300}{75}{0.3}{9}{2};
\draw[red] (0.6,-0.3) .. controls (0.6,.3) .. (0,1.35);
\draw[red, dashed] (0,1.35) .. controls (0,1.95) .. (0.6,3);
}\,,\ \ 
\tau^r_y =
\tikzmath[xscale=-1]{
\tubeguts{285}{60}{0.3}{9}{2};
\draw[red] (0.6,-0.3) .. controls (0.6,.3) .. (0,1.35);
\draw[red, dashed] (0,1.35) .. controls (0,1.95) .. (0.6,3);
}\,,
\]
where the red string denotes a module functor $y:\cM \to \cM$. The components of these natural transformations are given by

\begin{equation}\label{eq:tau_l}
(\tau^\ell_y)_1 = e\left( \bigoplus_{s \in \cM} s^\op\boxtimes ys \cong \bigoplus_{s,t} s^\op\boxtimes [t,ys]\otimes t \cong \bigoplus_{s,t} s^\op\otimes[y^*t,s]\boxtimes t \cong \bigoplus_{t} y^*t^\op\boxtimes t \right)e,
\end{equation}
and
\begin{multline}\label{eq:tau_r}
(\tau^r_y)_{m^\op\boxtimes n} = \Bigg([m,yn]_{\cX} \cong \bigoplus_{x \in \cX} [x\rhd m,yn]\otimes x \cong \bigoplus_x [y^*(x\rhd m),n]\otimes x \\ \cong \bigoplus_x [x\rhd(y^*m),n]\otimes x \cong [y^*m,n]_\cX \Bigg).
\end{multline}

It is immediate that these are both natural isomorphisms. They also satisfy compatibilities with the structures defined above.

\begin{lem}\label{lem:trace_cup}
The half traciators satisfy the equation
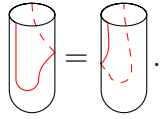
\[
\tikzmath{
\draw (0,1) arc (180:225:.3 and .15) coordinate (a);
\draw (0,1) arc (180:90:.3 and .15) coordinate (b);
\draw[red] (a) -- +(0,-.75) arc (-180:0:.15) to [out=90, in = 225] (.6,.5);
\draw[red,dashed] (.6,.5) to [out=135,in=270] (b);
\draw (0,0) arc (-180:0:.3) -- +(0,1) arc (0:540:.3 and .15) -- (0,0);
} =
\tikzmath{
\draw (0,1) arc (180:225:.3 and .15) coordinate (a);
\draw (0,1) arc (180:90:.3 and .15) coordinate (b);
\draw[red] (a) to [out=270,in=45] (0,.35);
\draw[red,dashed] (0,.35) to [out=135,in=270] (.1,.25) arc (-180:0:.15) to [out=90, in=270] (b);
\draw (0,0) arc (-180:0:.3) -- +(0,1) arc (0:540:.3 and .15) -- (0,0);
}\,.
\]
\end{lem}
\begin{proof}
This is equivalent to showing
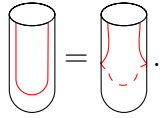
\[
\tikzmath{
\draw (0,1) arc (180:225:.3 and .15) coordinate (a) arc (225:315:.3 and .15) coordinate (b);
\draw[red] (a) -- +(0,-.75) arc (-180:0:.21) -- (b);
\draw (0,0) arc (-180:0:.3) -- +(0,1) arc (0:540:.3 and .15) -- (0,0);
} =
\tikzmath{
\draw (0,1) arc (180:225:.3 and .15) coordinate (a) arc (225:315:.3 and .15) coordinate (b);
\draw[red] (a) to [out=270, in=45] (0,.35);
\draw[red,dashed] (0,.35) to [out=135,in=270] (.1,.25) arc (-180:0:.19) to [out=90,in=225] (.6,.35);
\draw[red] (b) to [out=270, in=135] (.6,.35);
\draw (0,0) arc (-180:0:.3) -- +(0,1) arc (0:540:.3 and .15) -- (0,0);
}\,.
\]
It is enough to show that their components at $1_\cX$ are equal, since both are bimodule natural transformations. Moreover, since $\eta_{1_\cX}^\ell$ is defined as a map into $\bigoplus s^\op\boxtimes s$ followed by the idempotent $e$, it is sufficient to compute this equality on maps into this larger object, and then project into the subobject $S$.

We will show the Yoneda embeddings of these components are equal, i.e., that the induced maps $[x,1_\cX] \to [x,\bigoplus [s,yy^*s]_\cX]$ are equal for all $x \in \cX$. Recall that by (\ref{eq:eta_ell}), the component $\eta^\ell_{1_\cX}:1_\cX \to \bigoplus [s,s]_\cX$ corestricts to the summand at $1_\cX$. Since we are precomposing with $\eta^\ell_{1_\cX}$, we only need to consider the case $x = 1_\cX$, as all the other induced maps will be zero. 

The map on the left is just
\[1_\cX \xrightarrow{\eta^\ell_{1_\cX}} \bigoplus [s,s]_\cX \xrightarrow{[\id,\coev]} \bigoplus [s,yy^*x]_\cX,\]
which is represented at $1_\cX$ by
\[ \id_{1_\cX} \mapsto \bigoplus_s \id_s \otimes \id_{1_\cX} \mapsto \bigoplus_s \coev^y_s \otimes \id_{1_\cX}.
\]

For the other map, we begin by calculating the components of the traciators, when mapped into by $1_\cX$. First, recall that the component at $s^\op\boxtimes t$ of $\tau^r$ is given by
\[
[s,yt]_\cX = \bigoplus_x [x\lhd s,yt]\otimes x \cong \bigoplus_x [y^*(x\lhd s),t]\otimes x \cong \bigoplus_x [x\lhd y^*s,t]\otimes x = [y^*s,t]_\cX.
\]
This acts on maps from $1_\cX$ into this object by sending elements of $[s,yt]$ to elements of $[y^*s,t]$ via the adjunction, that is, by composing with the evaluation of the $y^*,y$ adjunction. Since the map in the graphical image above is actually $(\tau^r)^{-1}$, it is given by composing with the coevaluation of the adjunction. Similarly, for $\tau^\ell$, since we are restricting to the component at $1_\cX$, we can again notice by (\ref{eq:tau_l}) that the map is given by composing with the coevaluation.

Finally, notice that the cup on the back of the tube is actually $\ev^\op$, since left and right duals are switched in $\cM^\op$. Thus the map on the right is represented at $1_\cX$ by
\begin{align*}
\id_{1_\cX} &\mapsto \sum_s \id_s \\
 &\mapsto \bigoplus_s (\ev^y_s)^\op \otimes \id_{1_\cX} \\
 &\mapsto \bigoplus (\id_{y}\otimes \ev^y_s\id_{y^*s})\circ(\coev^y_s\otimes\id_{yy^*s})\circ(\coev^y_s)\otimes \id_{1_\cX} \\ 
&= \bigoplus \coev^y_s\otimes\id_{1_\cX},
\end{align*}
where the last equality is due to the snake equation for $y$.
\end{proof}

We can compose the half traciators $\tau_x^\ell$ and $(\tau_x^r)^{-1}$, represented graphically by
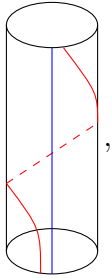
\[
\tikzmath{
    \draw (0,0) arc (180:-180:2*\x and \x);
    \draw (0,0) -- +(0,3) arc (180:-180:2*\x and \x);
    \draw (1.2,0) -- +(0,3);
    \draw[red] (0.45,-0.3) .. controls (0.45,.3) .. (0,.9);
    \draw[red, dashed] (0,.9) -- (1.2,1.7);
    \draw[red] (1.2,1.7) .. controls (1.2,2.1) .. (0.75,2.7);
    \draw[blue] (0.6,-0.3) -- (0.6,2.7);
    }\,,
\]
to define a natural isomorphism $\Tr(x\otimes y) \to \Tr(y\otimes x)$ (where $x$ denotes the red string).

\begin{prop}\label{prop:traciator}
    The composition of $\tau_x^\ell$ and $\tau_x^r$ defines a traciator; that is, for $x,y \in \End({}_\cX\cM)$, the map $\tau_{x, y}:\Tr(x\otimes y) \to \Tr(y \otimes x)$ represented above satisfies the condition $\tau_{x\otimes y,z} = \tau_{x,y\otimes z}\circ \tau_{y,x\otimes z}$ as in \cite{HPT2015}.
\end{prop}
\begin{proof}
It is enough to show that $\tau^\ell_{x\otimes y} = \tau^\ell_y \circ \tau^\ell_{x}$ and $\tau^r_{x\otimes y} = \tau^r_x\circ\tau^r_y$. For $\tau^\ell$, this follows from the definition in (\ref{eq:tau_l}) and the fact that the isomorphism in the adjunction $[t,xys] \cong [(xy)^*t,s]$ is the same as the composite $[t,xys] \cong [x^*t,ys] \cong [y^*x^*t,s] \cong [(xy)^*t,s]$. Similarly, for $\tau^r$, this follows from the compatibility of the $\cX$-module structures of $x$ and $y$ with composition of $x$ and $y$ and the compatibility of composition with the adjunction data. The result follows.
\end{proof}

We now prove the claim above, that the traciator given by composing the two half-traciators is the same as the natural one arising from the adjunction $\Phi \dashv \Tr$.

\begin{prop}\label{prop:same_traciator}
The traciator in Proposition \ref{prop:traciator} is the same as the one given in \cite{HPT2015}.
\end{prop}
\begin{proof}
The traciator in \cite{HPT2015} is given as the mate of the map
\[ \Phi(\Tr(xy)) \xrightarrow{\coev_y} \Phi(\Tr(xy))yy^* \xrightarrow{\beta} y\Phi(\Tr(xy))y^* \xrightarrow{\epsilon_{\Phi\dashv\Tr}} yxyy^* \xrightarrow{p} yxy^{**}y^* \xrightarrow{\ev_{y^*}} yx. \]
In the two dimensional graphical calculus, this means the traciator is the composition
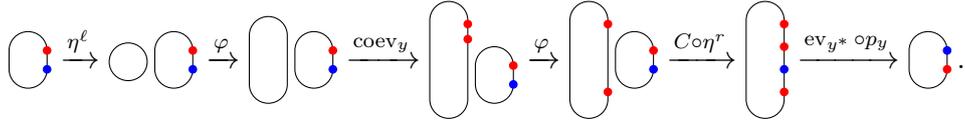
\[
\tikzmath{
\draw (0,0) arc (0:-180:.25) -- +(0,.25) arc (180:0:.25) -- (0,0);
\filldraw[blue] (0,0) circle (.05);
\filldraw[red] (0,.25) circle (.05);
}
\xrightarrow{\eta^\ell}
\tikzmath{
\draw (0,0) arc (0:-180:.25) -- +(0,.25) arc (180:0:.25) -- (0,0);
\draw (-.6,0.1) arc (0:360:.25);
\filldraw[blue] (0,0) circle (.05);
\filldraw[red] (0,.25) circle (.05);
}
\xrightarrow{\varphi}
\tikzmath{
\draw (0,0) arc (0:-180:.25) -- +(0,.25) arc (180:0:.25) -- (0,0);
\draw (-.6,-.2) arc (0:-180:.25) -- +(0,.65) arc (180:0:.25) -- (-.6,-.2);
\filldraw[blue] (0,0) circle (.05);
\filldraw[red] (0,.25) circle (.05);
}
\xrightarrow{\coev_y}
\tikzmath{
\draw (0,0) arc (0:-180:.25) -- +(0,.25) arc (180:0:.25) -- (0,0);
\draw (-.6,-.2) arc (0:-180:.25) -- +(0,1.05) arc (180:0:.25) -- (-.6,-.2);
\filldraw[blue] (0,0) circle (.05);
\filldraw[red] (0,.25) circle (.05);
\filldraw[red] (-.6,.6) circle (.05);
\filldraw[red] (-.6,.8) circle (.05);
}
\xrightarrow{\varphi}
\tikzmath{
\draw (0,0) arc (0:-180:.25) -- +(0,.25) arc (180:0:.25) -- (0,0);
\draw (-.6,-.4) arc (0:-180:.25) -- +(0,1.05) arc (180:0:.25) -- (-.6,-.4);
\filldraw[blue] (0,0) circle (.05);
\filldraw[red] (0,.25) circle (.05);
\filldraw[red] (-.6,.6) circle (.05);
\filldraw[red] (-.6,-.3) circle (.05);
}
\xrightarrow{C\circ\eta^r}
\tikzmath{
\draw (-.6,-.4) arc (0:-180:.25) -- +(0,1.05) arc (180:0:.25) -- (-.6,-.4);
\filldraw[blue] (-.6,0) circle (.05);
\filldraw[red] (-.6,.3) circle (.05);
\filldraw[red] (-.6,.6) circle (.05);
\filldraw[red] (-.6,-.3) circle (.05);
}
\xrightarrow{\ev_{y^*}\circ p_y}
\tikzmath{
\draw (0,0) arc (0:-180:.25) -- +(0,.25) arc (180:0:.25) -- (0,0);
\filldraw[red] (0,0) circle (.05);
\filldraw[blue] (0,.25) circle (.05);
}\,.
\]
To see that this is equal to the definition of the traciator above, consider the diagram below.
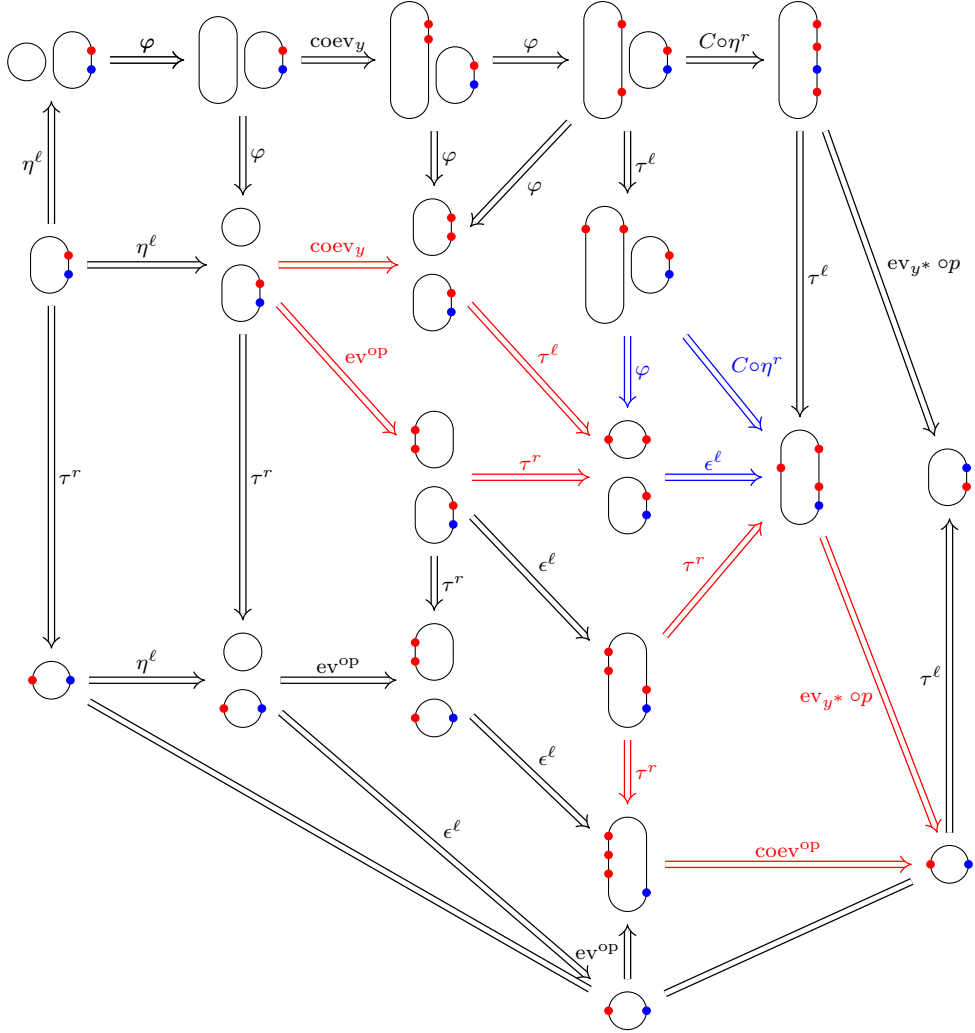
\[
\begin{tikzcd}
\tikzmath{
\draw (0,0) arc (0:-180:.25) -- +(0,.25) arc (180:0:.25) -- (0,0);
\draw (-.6,0.1) arc (0:360:.25);
\filldraw[blue] (0,0) circle (.05);
\filldraw[red] (0,.25) circle (.05);
}
\arrow[r, Rightarrow, "\varphi"]
\arrow[r, Rightarrow, "\varphi"]
&
\tikzmath{
\draw (0,0) arc (0:-180:.25) -- +(0,.25) arc (180:0:.25) -- (0,0);
\draw (-.6,-.2) arc (0:-180:.25) -- +(0,.65) arc (180:0:.25) -- (-.6,-.2);
\filldraw[blue] (0,0) circle (.05);
\filldraw[red] (0,.25) circle (.05);
}
\arrow[d, Rightarrow, "\varphi"]
\arrow[r, Rightarrow, "\coev_y"]
&
\tikzmath{
\draw (0,0) arc (0:-180:.25) -- +(0,.25) arc (180:0:.25) -- (0,0);
\draw (-.6,-.2) arc (0:-180:.25) -- +(0,1.05) arc (180:0:.25) -- (-.6,-.2);
\filldraw[blue] (0,0) circle (.05);
\filldraw[red] (0,.25) circle (.05);
\filldraw[red] (-.6,.6) circle (.05);
\filldraw[red] (-.6,.8) circle (.05);
}
\arrow[r, Rightarrow, "\varphi"]
\arrow[d, Rightarrow, "\varphi"]
&
\tikzmath{
\draw (0,0) arc (0:-180:.25) -- +(0,.25) arc (180:0:.25) -- (0,0);
\draw (-.6,-.4) arc (0:-180:.25) -- +(0,1.05) arc (180:0:.25) -- (-.6,-.4);
\filldraw[blue] (0,0) circle (.05);
\filldraw[red] (0,.25) circle (.05);
\filldraw[red] (-.6,.6) circle (.05);
\filldraw[red] (-.6,-.3) circle (.05);
}
\arrow[r, Rightarrow, "C\circ \eta^r"]
\arrow[d, Rightarrow, "\tau^\ell"]
\arrow[dl, Rightarrow, "\varphi"]
&
\tikzmath{
\draw (-.6,-.4) arc (0:-180:.25) -- +(0,1.05) arc (180:0:.25) -- (-.6,-.4);
\filldraw[blue] (-.6,0) circle (.05);
\filldraw[red] (-.6,.3) circle (.05);
\filldraw[red] (-.6,.6) circle (.05);
\filldraw[red] (-.6,-.3) circle (.05);
}
\arrow[ddr, Rightarrow, "\ev_{y^*}\circ p"]
\arrow[dd, Rightarrow, "\tau^\ell"]
\\
\tikzmath{
\draw (0,0) arc (0:-180:.25) -- +(0,.25) arc (180:0:.25) -- (0,0);
\filldraw[blue] (0,0) circle (.05);
\filldraw[red] (0,.25) circle (.05);
}
\arrow[u, Rightarrow, "\eta^\ell"]
\arrow[r, Rightarrow, "\eta^\ell"]
\arrow[dd, Rightarrow, "\tau^r"]
&
\tikzmath{
\draw (0,0) arc (0:-180:.25) -- +(0,.25) arc (180:0:.25) -- (0,0);
\draw (0,1) arc (0:360:.25);
\filldraw[blue] (0,0) circle (.05);
\filldraw[red] (0,.25) circle (.05);
}
\arrow[red,rd, Rightarrow, "\ev^\op"]
\arrow[dd, Rightarrow, "\tau^r"]
\arrow[red,r, Rightarrow, "\coev_y"]
&
\tikzmath{
\draw (0,0) arc (0:-180:.25) -- +(0,.25) arc (180:0:.25) -- (0,0);
\draw (0,1) arc (0:-180:.25) -- +(0,.25) arc (180:0:.25) -- (0,1);
\filldraw[blue] (0,0) circle (.05);
\filldraw[red] (0,.25) circle (.05);
\filldraw[red] (0,1) circle (.05);
\filldraw[red] (0,1.25) circle (.05);
}
\arrow[red,dr, Rightarrow, "\tau^\ell"]
&
\tikzmath{
\draw (0,0) arc (0:-180:.25) -- +(0,.25) arc (180:0:.25) -- (0,0);
\draw (-.6,-.4) arc (0:-180:.25) -- +(0,1.05) arc (180:0:.25) -- (-.6,-.4);
\filldraw[blue] (0,0) circle (.05);
\filldraw[red] (0,.25) circle (.05);
\filldraw[red] (-.6,.6) circle (.05);
\filldraw[red] (-1.1,.6) circle (.05);
}
\arrow[blue,dr, Rightarrow, "C\circ \eta^r"]
\arrow[blue,d, Rightarrow, "\varphi"]
\\
&
&
\tikzmath{
\draw (0,0) arc (0:-180:.25) -- +(0,.25) arc (180:0:.25) -- (0,0);
\draw (0,1) arc (0:-180:.25) -- +(0,.25) arc (180:0:.25) -- (0,1);
\filldraw[blue] (0,0) circle (.05);
\filldraw[red] (0,.25) circle (.05);
\filldraw[red] (-.5,1) circle (.05);
\filldraw[red] (-.5,1.25) circle (.05);
}
\arrow[d, Rightarrow, "\tau^r"]
\arrow[red,r, Rightarrow, "\tau^r"]
\arrow[dr, Rightarrow, "\epsilon^\ell"]
&
\tikzmath{
\draw (0,0) arc (0:-180:.25) -- +(0,.25) arc (180:0:.25) -- (0,0);
\draw (0,1) arc (0:-180:.25) arc (180:0:.25);
\filldraw[blue] (0,0) circle (.05);
\filldraw[red] (0,.25) circle (.05);
\filldraw[red] (0,1) circle (.05);
\filldraw[red] (-.5,1) circle (.05);
}
\arrow[blue,r, Rightarrow, "\epsilon^\ell"]
&
\tikzmath{
\draw (0,0) arc (0:-180:.25) -- +(0,.25) -- +(0,.75) arc (180:0:.25) -- (0,0);
\filldraw[blue] (0,0) circle (.05);
\filldraw[red] (-.5,.5) circle (.05);
\filldraw[red] (0,.25) circle (.05);
\filldraw[red] (0,.75) circle (.05);
}
\arrow[red,ddr, Rightarrow, "\ev_{y^*}\circ p"']
&
\tikzmath{
\draw (0,0) arc (0:-180:.25) -- +(0,.25) arc (180:0:.25) -- (0,0);
\filldraw[red] (0,0) circle (.05);
\filldraw[blue] (0,.25) circle (.05);
}
\\
\tikzmath{
\draw (0,0) arc (0:360:.25);
\filldraw[blue] (0,0) circle (.05);
\filldraw[red] (-.5,0) circle (.05);
}
\arrow[r, Rightarrow, "\eta^\ell"]
\arrow[ddrrr, equals]
&
\tikzmath{
\draw (0,0) arc (0:360:.25);
\draw (0,.75) arc (0:360:.25);
\filldraw[blue] (0,0) circle (.05);
\filldraw[red] (-.5,0) circle (.05);
}
\arrow[r, Rightarrow, "\ev^\op"]
\arrow[ddrr, Rightarrow, "\epsilon^\ell"]
&
\tikzmath{
\draw (0,0) arc (0:360:.25);
\draw (0,.75) arc (0:-180:.25) -- +(0,.25) arc (180:0:.25) -- (0,.75);
\filldraw[red] (-.5,.75) circle (.05);
\filldraw[red] (-.5,1) circle (.05);
\filldraw[blue] (0,0) circle (.05);
\filldraw[red] (-.5,0) circle (.05);
}
\arrow[dr, Rightarrow, "\epsilon^\ell"]
&
\tikzmath{
\draw (0,0) arc (0:-180:.25) -- +(0,.25) -- +(0,.75) arc (180:0:.25) -- (0,0);
\filldraw[blue] (0,0) circle (.05);
\filldraw[red] (0,.25) circle (.05);
\filldraw[red] (-.5,.5) circle (.05);
\filldraw[red] (-.5,.75) circle (.05);
}
\arrow[red,d, Rightarrow, "\tau^r"]
\arrow[red,ur, Rightarrow, "\tau^r"]
\\
&&&
\tikzmath{
\draw (0,0) arc (0:-180:.25) -- +(0,.25) -- +(0,.75) arc (180:0:.25) -- (0,0);
\filldraw[blue] (0,0) circle (.05);
\filldraw[red] (-.5,.25) circle (.05);
\filldraw[red] (-.5,.5) circle (.05);
\filldraw[red] (-.5,.75) circle (.05);
}
\arrow[red,rr, Rightarrow, "\coev^\op"]
&&
\tikzmath{
\draw (0,0) arc (0:360:.25);
\filldraw[blue] (0,0) circle (.05);
\filldraw[red] (-.5,0) circle (.05);
}
\arrow[dll, equals]
\arrow[uu, Rightarrow, "\tau^\ell"]
\\
&&&
\tikzmath{
\draw (0,0) arc (0:360:.25);
\filldraw[blue] (0,0) circle (.05);
\filldraw[red] (-.5,0) circle (.05);
}
\arrow[u, Rightarrow, "\ev^\op"]
\end{tikzcd}
\]
(Here, $\varphi$ represents the interchanger). Most of the faces in this diagram are naturality squares or snake equations. The red squares are both applications of Lemma \ref{lem:trace_cup}, and the blue triangle is an application of Proposition \ref{prop:mult}. Therefore, the two paths around the outside of the diagram are equal. One is the traciator defined in \cite{HPT2015}, while the other is $\tau_\ell\circ\tau_r$, as defined above.
\end{proof}

As a result of Proposition \ref{prop:same_traciator}, we obtain coherence results for the traciators and the braiding and twist isomorphisms, as in \cite{HPT2015}. However, we provide alternate proofs below using the graphical calculus within the 3-category. This demonstrates the usefulness of the 3-categorical approach.

First, we show compatibility with the braiding. Graphically, the compatibility with the braiding is given by the equation

\begin{equation}\label{eq:braid_axiom}
\tikzmath{
\draw (0,0) arc (-180:0:2*\x and \x) arc (180:0:3*\x) arc (-180:0:2*\x and \x) arc (0:60:6*\x) arc (240:180:{4*\x})
arc (0:540:2*\x and \x) %this is the wrap around up top
arc (0:-60:4*\x) arc (120:180:6*\x);
\draw[dashed] (0,0) arc (180:0:2*\x and \x);
\draw[dashed] (10*\x,0) arc (180:0:2*\x and \x);
\draw[red] (2*\x,-\x) arc (180:85:2*\x);
\draw[red, dashed] ({4*\x+2*\x*cos(85)},{-\x+2*\x*sin(95)}) .. controls (1,1) .. (3,1.8);
\draw[red] (3,1.8) -- +(-.2,0) arc (-90:-180:.5);
\draw[blue] (12*\x,-\x) arc (0:60:6*\x) -- +(-{cos(30)*.4},{.5*.4}) arc (240:180:{3*\x}) -- +(0,.1);
} =
\tikzmath{
\draw (0,0) arc (60:-60:3*\x) arc (120:180:9*\x) arc (-180:0:2*\x and \x) arc (180:120:6*\x) arc (-60:60:6*\x) arc (240:180:\x) -- +(0,.2) arc (0:540:2*\x and \x) -- +(0,-.2) arc (0:-60:\x);
\draw[red] (-2.5*\x,-{(sqrt(3)*3*\x + sqrt(3)*9*\x/2)}-\x) arc (180:120:9*\x) arc (-60:60:4*\x) arc (240:180:2.9*\x);
\draw[fill=white] (0,0) arc (120:240:3*\x) arc (60:0:9*\x) arc (0:-180:2*\x and \x) arc (0:60:6*\x) arc (240:120:6*\x);
\draw[blue] (2.5*\x,-{(sqrt(3)*3*\x + sqrt(3)*9*\x/2)}-\x) arc (0:60:9*\x) arc (240:120:4*\x) arc (-60:0:2.9*\x);
}.
\end{equation}

\begin{prop}\label{prop:braid_axiom}
The half-traciators $\tau_\ell$ and $\tau_r$ satisfy (\ref{eq:braid_axiom}).
\end{prop}
\begin{proof}
We utilize the graphical calculus of string diagrams. The map on the left is given by
\[
\tikzmath{
\draw (0,0) arc (0:360:.25);
\draw (0,.75) arc (0:360:.25);
\filldraw[red] (0,0) circle (.05);
\filldraw[blue] (0,.75) circle (.05);
}
\ \xRightarrow{\tau_r}
\tikzmath{
\draw (0,0) arc (0:360:.25);
\draw (0,.75) arc (0:360:.25);
\filldraw[red] (-.5,0) circle (.05);
\filldraw[blue] (0,.75) circle (.05);
}
\ \xRightarrow{\epsilon}\ 
\tikzmath{
\draw (0,0) arc (0:-180:.25) -- +(0,.25) arc (180:0:.25) -- (0,0);
\filldraw[blue] (0,.25) circle (.05);
\filldraw[red] (-.5,0) circle (.05);
}
\ \xRightarrow{\phi}\ 
\tikzmath{
\draw (0,0) arc (0:-180:.25) -- +(0,.25) arc (180:0:.25) -- (0,0);
\filldraw[blue] (0,0) circle (.05);
\filldraw[red] (-.5,.25) circle (.05);
}
\ \xRightarrow{\tau_r^{-1}}\ 
\tikzmath{
\draw (0,0) arc (0:-180:.25) -- +(0,.25) arc (180:0:.25) -- (0,0);
\filldraw[blue] (0,0) circle (.05);
\filldraw[red] (0,.25) circle (.05);
}\,,
\]
while the map on the right is
\[
\tikzmath{
\draw (0,0) arc (0:360:.25);
\draw (0,.75) arc (0:360:.25);
\filldraw[red] (0,0) circle (.05);
\filldraw[blue] (0,.75) circle (.05);
}
\ \xRightarrow{\phi}\ 
\tikzmath{
\draw (0,0) arc (0:360:.25);
\draw (-.75,0) arc (0:360:.25);
\filldraw[blue] (0,0) circle (.05);
\filldraw[red] (-.75,0) circle (.05);
}
\ \xRightarrow{\phi}\ 
\tikzmath{
\draw (0,0) arc (0:360:.25);
\draw (0,.75) arc (0:360:.25);
\filldraw[blue] (0,0) circle (.05);
\filldraw[red] (0,.75) circle (.05);
}
\ \xRightarrow{\epsilon}\ 
\tikzmath{
\draw (0,0) arc (0:-180:.25) -- +(0,.25) arc (180:0:.25) -- (0,0);
\filldraw[blue] (0,0) circle (.05);
\filldraw[red] (0,.25) circle (.05);
}\,.
\]
To verify these are equal, consider the following pasting diagram:
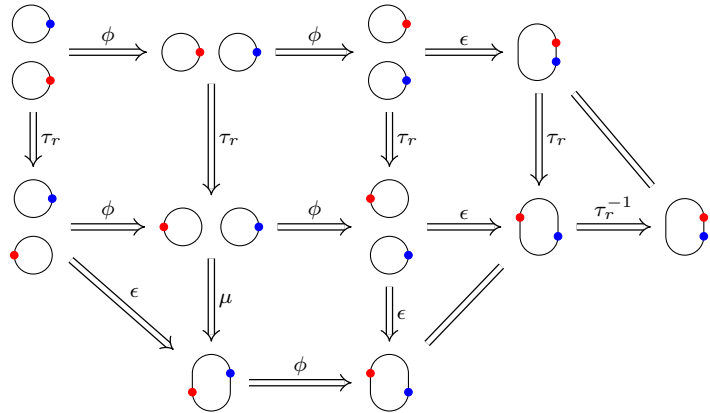
\[
\begin{tikzcd}
\tikzmath{
\draw (0,0) arc (0:360:.25);
\draw (0,.75) arc (0:360:.25);
\filldraw[red] (0,0) circle (.05);
\filldraw[blue] (0,.75) circle (.05);
}
\arrow[r, Rightarrow, "\phi"]
\arrow[d, Rightarrow, "\tau_r"]
&
\tikzmath{
\draw (0,0) arc (0:360:.25);
\draw (-.75,0) arc (0:360:.25);
\filldraw[blue] (0,0) circle (.05);
\filldraw[red] (-.75,0) circle (.05);
}
\arrow[r, Rightarrow, "\phi"]
\arrow[d, Rightarrow, "\tau_r"]
&
\tikzmath{
\draw (0,0) arc (0:360:.25);
\draw (0,.75) arc (0:360:.25);
\filldraw[blue] (0,0) circle (.05);
\filldraw[red] (0,.75) circle (.05);
}
\arrow[r, Rightarrow, "\epsilon"]
\arrow[d, Rightarrow, "\tau_r"]
&
\tikzmath{
\draw (0,0) arc (0:-180:.25) -- +(0,.25) arc (180:0:.25) -- (0,0);
\filldraw[blue] (0,0) circle (.05);
\filldraw[red] (0,.25) circle (.05);
}
\arrow[d, Rightarrow, "\tau_r"]
\arrow[dr, Rightarrow, equal]
\\
\tikzmath{
\draw (0,0) arc (0:360:.25);
\draw (0,.75) arc (0:360:.25);
\filldraw[red] (-.5,0) circle (.05);
\filldraw[blue] (0,.75) circle (.05);
}
\arrow[r, Rightarrow, "\phi"]
\arrow[dr, Rightarrow, "\epsilon"]
&
\tikzmath{
\draw (0,0) arc (0:360:.25);
\draw (-.75,0) arc (0:360:.25);
\filldraw[blue] (0,0) circle (.05);
\filldraw[red] (-1.25,0) circle (.05);
}
\arrow[d, Rightarrow, "\mu"]
\arrow[r, Rightarrow, "\phi"]
&
\tikzmath{
\draw (0,0) arc (0:360:.25);
\draw (0,.75) arc (0:360:.25);
\filldraw[blue] (0,0) circle (.05);
\filldraw[red] (-.5,.75) circle (.05);
}
\arrow[d, Rightarrow, "\epsilon"]
\arrow[r, Rightarrow, "\epsilon"]
&
\tikzmath{
\draw (0,0) arc (0:-180:.25) -- +(0,.25) arc (180:0:.25) -- (0,0);
\filldraw[blue] (0,0) circle (.05);
\filldraw[red] (-.5,.25) circle (.05);
}
\arrow[r, Rightarrow, "\tau_r^{-1}"]
&
\tikzmath{
\draw (0,0) arc (0:-180:.25) -- +(0,.25) arc (180:0:.25) -- (0,0);
\filldraw[blue] (0,0) circle (.05);
\filldraw[red] (0,.25) circle (.05);
}
\\
& \tikzmath{
\draw (0,0) arc (0:-180:.25) -- +(0,.25) arc (180:0:.25) -- (0,0);
\filldraw[blue] (0,.25) circle (.05);
\filldraw[red] (-.5,0) circle (.05);
}
\arrow[r, Rightarrow, "\phi"]
&
\tikzmath{
\draw (0,0) arc (0:-180:.25) -- +(0,.25) arc (180:0:.25) -- (0,0);
\filldraw[blue] (0,0) circle (.05);
\filldraw[red] (-.5,.25) circle (.05);
}
\arrow[ur, equal]
\end{tikzcd}
\]
Here the map labeled $\mu$ is the multiplication mentioned in Remark \ref{rem:thirdmult}; the triangle and square containing $\mu$ commute by the remark. The other faces are naturality squares or trivial, so this completes the proof.
\end{proof}

In addition, there is a compatibility between the traciator and the twist, given by

\begin{equation}\label{eq:axiom1}
    \tikzmath{
    \draw (0,0) arc (180:-180:2*\x and \x);
    \draw (0,0) -- +(0,3) arc (180:-180:2*\x and \x);
    \draw (1.2,0) -- +(0,3);
    \draw[red] (0.6,-0.3) .. controls (0.6,.3) .. (0,.9);
    \draw[red, dashed] (0,.9) -- (1.2,1.7);
    \draw[red] (1.2,1.7) .. controls (1.2,2.1) .. (0.6,2.7);
    } =
    \theta_{\Tr(y)},
\end{equation}
where $\theta$ is the natural balancing arising from the pivotality and braiding, and $y$ is the red string.

\begin{prop}\label{prop:axiom1}
    The traciator given in Proposition \ref{prop:traciator} satisfies Equation \ref{eq:axiom1}.
\end{prop}
\begin{proof}
Recall that the balancing arising from a pivotal braided category is given by $(\ev_x\otimes p_{x^{**}})\circ(\id_x\otimes\beta)\circ(\id_x\otimes \coev_{x^*})$, where $p$ is the pivotality isomorphism $x^{**}\cong x$. As $\Tr(y)$ is a composite of three functors, its evaluation and coevaluation maps are given by composites of the evaluation and coevaluations for its parts. In the following diagram, this composite goes around the top and right edges.

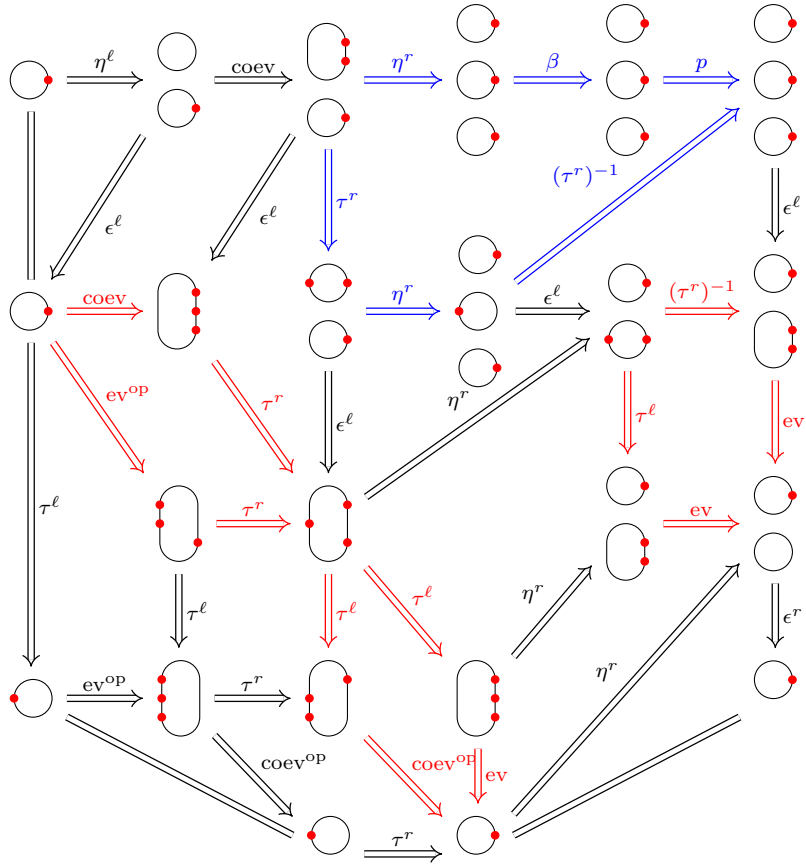
\begin{center}
\begin{tikzcd}
%1
\tikzmath{
\draw (0,0) arc (0:360:.25);
\filldraw[red] (0,0) circle (.05);
}
\arrow[r, Rightarrow, "\eta^\ell"]
\arrow[d, equals]
&
%2
\tikzmath{
\draw (0,0) arc (0:360:.25);
\draw (0,.75) arc (0:360:.25);
\filldraw[red] (0,0) circle (.05);
}
\arrow[dl, Rightarrow, "\epsilon^\ell"]
\arrow[r, Rightarrow, "\coev"]
&
%3
\tikzmath{
\draw (0,0) arc (0:360:.25);
\draw (0,.75) -- +(0,.25) arc (0:180:.25) -- +(0,-.25) arc (-180:0:.25);
\filldraw[red] (0,0) circle (.05);
\filldraw[red] (0,.75) circle (.05);
\filldraw[red] (0,1) circle (.05);
}
\arrow[r,blue, Rightarrow, "\eta^r"]
\arrow[dl, Rightarrow, "\epsilon^\ell"]
\arrow[d,blue, Rightarrow, "\tau^r"]
&
%4
\tikzmath{
\draw (0,0) arc (0:360:.25);
\draw (0,.75) arc (0:360:.25);
\draw (0,1.5) arc (0:360:.25);
\filldraw[red] (0,0) circle (.05);
\filldraw[red] (0,.75) circle (.05);
\filldraw[red] (0,1.5) circle (.05);
}
\arrow[r,blue, Rightarrow, "\beta"]
&
%5
\tikzmath{
\draw (0,0) arc (0:360:.25);
\draw (0,.75) arc (0:360:.25);
\draw (0,1.5) arc (0:360:.25);
\filldraw[red] (0,0) circle (.05);
\filldraw[red] (0,.75) circle (.05);
\filldraw[red] (0,1.5) circle (.05);
}
\arrow[r,blue, Rightarrow, "p"]
&
%6
\tikzmath{
\draw (0,0) arc (0:360:.25);
\draw (0,.75) arc (0:360:.25);
\draw (0,1.5) arc (0:360:.25);
\filldraw[red] (0,0) circle (.05);
\filldraw[red] (0,.75) circle (.05);
\filldraw[red] (0,1.5) circle (.05);
}
\arrow[d, Rightarrow, "\epsilon^\ell"]
\\
%7
\tikzmath{
\draw (0,0) arc (0:360:.25);
\filldraw[red] (0,0) circle (.05);
}
\arrow[r,red, Rightarrow, "\coev"]
\arrow[dd, Rightarrow, "\tau^\ell"]
\arrow[dr,red, Rightarrow, "\ev^\op"]
&
%8
\tikzmath{
\draw (0,0) -- +(0,.5) arc (0:180:.25) -- +(0,-.5) arc (-180:0:.25);
\filldraw[red] (0,0) circle (.05);
\filldraw[red] (0,.25) circle (.05);
\filldraw[red] (0,.5) circle (.05);
}
\arrow[dr,red, Rightarrow, "\tau^r"]
&
%9
\tikzmath{
\draw (0,0) arc (0:360:.25);
\draw (0,.75) arc (0:360:.25);
\filldraw[red] (0,0) circle (.05);
\filldraw[red] (0,.75) circle (.05);
\filldraw[red] (-.5,.75) circle (.05);
}
\arrow[r,blue, Rightarrow, "\eta^r"]
\arrow[d, Rightarrow, "\epsilon^\ell"]
&
%10
\tikzmath{
\draw (0,0) arc (0:360:.25);
\draw (0,.75) arc (0:360:.25);
\draw (0,1.5) arc (0:360:.25);
\filldraw[red] (0,0) circle (.05);
\filldraw[red] (-.5,.75) circle (.05);
\filldraw[red] (0,1.5) circle (.05);
}
\arrow[urr,blue, Rightarrow, "(\tau^r)^{-1}"]
\arrow[r, Rightarrow, "\epsilon^\ell"]
&
%14
\tikzmath{
\draw (0,0) arc (0:360:.25);
\draw (0,.75) arc (0:360:.25);
\filldraw[red] (0,0) circle (.05);
\filldraw[red] (0,.75) circle (.05);
\filldraw[red] (-.5,0) circle (.05);
}
\arrow[r,red, Rightarrow, "(\tau^r)^{-1}"]
\arrow[d,red, Rightarrow, "\tau^\ell"]
&
%11
\tikzmath[yscale=-1]{
\draw (0,0) arc (0:360:.25);
\draw (0,.75) -- +(0,.25) arc (0:180:.25) -- +(0,-.25) arc (-180:0:.25);
\filldraw[red] (0,0) circle (.05);
\filldraw[red] (0,.75) circle (.05);
\filldraw[red] (0,1) circle (.05);
}
\arrow[d,red, Rightarrow, "\ev"]
\\
&
%12
\tikzmath{
\draw (0,0) -- +(0,.5) arc (0:180:.25) -- +(0,-.5) arc (-180:0:.25);
\filldraw[red] (0,0) circle (.05);
\filldraw[red] (-.5,.25) circle (.05);
\filldraw[red] (-.5,.5) circle (.05);
}
\arrow[r,red, Rightarrow, "\tau^r"]
\arrow[d, Rightarrow, "\tau^\ell"]
&
%13
\tikzmath{
\draw (0,0) -- +(0,.5) arc (0:180:.25) -- +(0,-.5) arc (-180:0:.25);
\filldraw[red] (0,0) circle (.05);
\filldraw[red] (-.5,.25) circle (.05);
\filldraw[red] (0,.5) circle (.05);
}
\arrow[urr, Rightarrow, "\eta^r"]
\arrow[d,red, Rightarrow, "\tau^\ell"]
\arrow[dr,red, Rightarrow, "\tau^\ell"]
&&
%20
\tikzmath[yscale=-1]{
\draw (0,0) arc (0:360:.25);
\draw (0,.75) -- +(0,.25) arc (0:180:.25) -- +(0,-.25) arc (-180:0:.25);
\filldraw[red] (0,0) circle (.05);
\filldraw[red] (0,.75) circle (.05);
\filldraw[red] (0,1) circle (.05);
}
\arrow[r,red, Rightarrow, "\ev"]
&
%15
\tikzmath{
\draw (0,0) arc (0:360:.25);
\draw (0,.75) arc (0:360:.25);
\filldraw[red] (0,.75) circle (.05);
}
\arrow[d, Rightarrow, "\epsilon^r"]
\\
%16
\tikzmath{
\draw (0,0) arc (0:360:.25);
\filldraw[red] (-.5,0) circle (.05);
}
\arrow[r, Rightarrow, "\ev^\op"]
\arrow[drr, equals]
&
%17
\tikzmath{
\draw (0,0) -- +(0,.5) arc (0:180:.25) -- +(0,-.5) arc (-180:0:.25);
\filldraw[red] (-.5,0) circle (.05);
\filldraw[red] (-.5,.25) circle (.05);
\filldraw[red] (-.5,.5) circle (.05);
}
\arrow[r, Rightarrow, "\tau^r"]
\arrow[dr, Rightarrow, "\coev^\op"]
&
%18
\tikzmath{
\draw (0,0) -- +(0,.5) arc (0:180:.25) -- +(0,-.5) arc (-180:0:.25);
\filldraw[red] (-.5,0) circle (.05);
\filldraw[red] (-.5,.25) circle (.05);
\filldraw[red] (0,.5) circle (.05);
}
\arrow[dr,red, Rightarrow, "\coev^\op"]
&
%19
\tikzmath{
\draw (0,0) -- +(0,.5) arc (0:180:.25) -- +(0,-.5) arc (-180:0:.25);
\filldraw[red] (0,0) circle (.05);
\filldraw[red] (0,.25) circle (.05);
\filldraw[red] (0,.5) circle (.05);
}
\arrow[ur, Rightarrow, "\eta^r"]
\arrow[d, red, Rightarrow, "\ev"]
&
&
%22
\tikzmath{
\draw (0,.75) arc (0:360:.25);
\filldraw[red] (0,.75) circle (.05);
}
\arrow[dll, equals]
\\
&&
%23
\tikzmath[xscale=-1]{
\draw (0,.75) arc (0:360:.25);
\filldraw[red] (0,.75) circle (.05);
}
\arrow[r, Rightarrow, "\tau^r"]
&
%24
\tikzmath{
\draw (0,.75) arc (0:360:.25);
\filldraw[red] (0,.75) circle (.05);
}
\arrow[uurr, Rightarrow, "\eta^r"]
\end{tikzcd}
\end{center}
Here, the red squares commute by (variants of) Lemma \ref{lem:trace_cup}, and the blue polygon commutes by Proposition \ref{prop:braid_axiom}. The map around the top is the twist automorphism, as discussed above, while the map around the left and bottom is the trace map on the left side of (\ref{eq:axiom1}). 
\end{proof}

%%%%%%%%%%%%%%%%%%%%%%%%%%%%%%%%%%%%%%%%%%%%%%%%%%%%%%%%%%%%%%%%

\subsection{Equivalence with anchored planar algebras}\label{sec:equiv}

In \cite{HPT2016}, it is shown that there is a correspondence between anchored planar algebras in a braided pivotal category $\cB$ and pointed pivotal module tensor categories over $\cB$. The construction above clearly corresponds to a pointed pivotal module tensor category: choose any self-dual $y \in \End(\cM)$, and take the pointed pivotal module tensor subcategory generated by this object. Moreover, it is clear that any pointed pivotal module tensor category with a self-dual generator can arise this way by Lemma \ref{lem:setup}. The content of Propositions \ref{prop:mult} and \ref{prop:same_traciator} shows that the generating maps agree between these two constructions. It remains to be seen that the above construction defines an anchored planar algebra.

\begin{thm}[Theorem \ref{thm:3dapa}]\label{thm:planar-alg-in-3-cat}
    Suppose we have a planar pivotal 3-category $\cC$, objects $A,B$, morphisms $M:A \to B$ and $M^*:B\to A$, and adjoint pair $F:M\boxtimes_B M^* \to \id_A$, $F^*:\id_A \to M\boxtimes_B M^*$ with isomorphisms
    \[ \tau^\ell:(\id_M\boxtimes -)\circ F \to (-^*\boxtimes \id_M^*)\circ F,\ \ \tau^r:F^*\circ(\id_M\boxtimes -) \to F^*\circ(-^*\boxtimes \id_M^*), \]
    satisfying the compatibilities in
%Lemma \ref{lem:trace_cup} and
equations (\ref{eq:braid_axiom}) and (\ref{eq:axiom1}). Then for every self-dual endomorphism $x$ in $\End(M)$, we can construct an anchored planar algebra internal to $\End(\id_A)$ which interprets anchored planar tangles via the graphical calculus of $\cC$.
\end{thm}
\begin{proof}
In \cite{HPT2016}, the authors provide a small list of generators and relations that are equivalent to those for an anchored planar algebra. Given a braided pivotal category $\cB$, an anchored planar algebra inside $\cB$ is the same as a list of objects $\cP[i] \in \cB$ for each $i$ and generators
\[ \eta:1 \to \cP[0]\qquad \alpha_i:\cP[n+2] \to \cP[n]\qquad \tilde\alpha_i:\cP[n] \to \cP[n+2] \qquad \omega_{i,j}:\cP[n]\otimes \cP[j] \to \cP[n+j] \]
satisfying a list of 9 axioms (C1)$-$(C9).
As $\cC$ is a planar pivotal 3-category, $\End(\id_A)$ is a braided pivotal category. We choose $\cP[i] = F\circ x^n\circ F^*$. The generator $\eta$ is given by the unit of the adjunction. The generators $\alpha_i$ and $\tilde\alpha_i$ are the cup and cap for $x$ (as $x$ is self-dual), evaluated at the $i$th position. The generator $\omega_{i,j}$ is given graphically by the diagram
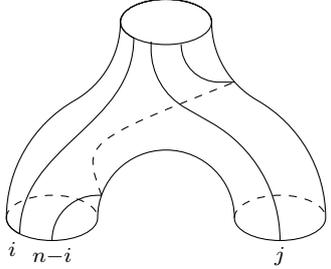
\[
\tikzmath{
\draw (0,0) arc (-180:0:2*\x and \x) arc (180:0:3*\x) arc (-180:0:2*\x and \x) arc (0:60:6*\x) arc (240:180:{4*\x})
arc (0:540:2*\x and \x) coordinate (b) %this is the wrap around up top
arc (0:-60:4*\x) arc (120:180:6*\x);
\draw (0,0) arc (180:225:2*\x and \x) coordinate (a);
\draw (b) arc (180:225:2*\x and \x) coordinate (c);
\draw (a) to [out=90,in=270] (c);
\draw[draw=none] (a) -- +(-.1,-.2) coordinate (d);
\node at (d) {$\scriptstyle i$};
\draw[dashed] (0,0) arc (180:0:2*\x and \x);
\draw[dashed] (10*\x,0) arc (180:0:2*\x and \x);
\draw (2*\x,-\x) arc (180:85:2*\x);
\draw[dashed] ({4*\x+2*\x*cos(85)},{-\x+2*\x*sin(95)}) .. controls (1,1) .. (3,1.8);
\draw (3,1.8) -- +(-.2,0) arc (-90:-180:.5);
\draw (12*\x,-\x) arc (0:60:6*\x) -- +(-{cos(30)*.4},{.5*.4}) arc (240:180:{3*\x}) -- +(0,.1);
\node at (12*\x,-.5) {$\scriptstyle j$};
\node at (2*\x,-.5) {$\scriptstyle n-i$};
}\,.
\]
Now many of the axioms follow immediately. Axiom (C1) states that $\eta$ is a unit for the multiplication $\omega_{0,j}$ and $\omega_{i,0}$, which follows immediately from the snake equations for the adjunction between $F$ and $F^*$. Axioms (C2), (C3), (C5), and (C6) follow from naturality of cups and caps for $x$ and naturality of the half-traciators. Axiom (C4) follows from naturality of cups and caps and the snake equations. Finally, axiom (C7) is an associativity condition for the multiplication, (C8) is the compatibility with the braiding, and (C9) is the compatibility with the twist. The associativity condition is immediate by properties of adjunctions, and (C8) and (C9) follow from (\ref{eq:braid_axiom}) and (\ref{eq:axiom1}), respectively. Thus we obtain an anchored planar algebra, as claimed.
\end{proof}

\begin{thm}[Theorem \ref{thm:main}]
All anchored planar algebras internal to braided fusion categories can be obtained via the construction in Theorem \ref{thm:planar-alg-in-3-cat}.
\end{thm}
\begin{proof}
Given an anchored planar algebra $\cP$ in $\cB \cong \End^\cA(\cX)$, by \cite{HPT2016} it corresponds to some pointed pivotal module tensor category $\cY \cong \End^\cA(_\cX\cM)$. Following the construction in Section \ref{sec:const}, we obtain maps satisfying the conditions in Theorem \ref{thm:planar-alg-in-3-cat}, and thus an anchored planar algebra. By Propositions \ref{prop:mult} and \ref{prop:same_traciator}, the generating maps of this anchored planar algebra agree with the generating maps of the construction in \cite{HPT2016}, and so this anchored planar algebra is isomorphic to $\cP$.
\end{proof}

%%%%%%%%%%%%%%%%%%%%%%%%%%%%%%%%%%%%%%%%%%%%%%%%%%%%%%%%%%%%%%%%

\subsection{Application to rigidity}\label{sec:apps}

The advantage of this construction is that it greatly expands the graphical language of anchored planar algebras. The shapes defined in Section \ref{sec:const} can be used to build and interpret surfaces of higher genus, extending the anchored planar operad. Additionally, abstract categorical tools can be applied to constructions in any anchored planar algebra.

One question this approach answers immediately has to do with the duality pairing of objects of the form $\Tr(y)$. In \cite{HPTunitary}, it was shown that the multiplication gives a duality pairing in the unitary setting, but the general question asked in \cite[Remark 5.4]{HPT2015} was left open. Here, we can answer it.

\begin{cor}
The rigidity of $\Tr(y)$ is witnessed by the adjunction data above; that is,
\[ 
\tikzmath{
\draw (0,0) arc (-180:0:2*\x and \x) arc (180:0:3*\x) arc (-180:0:2*\x and \x) arc (0:60:6*\x) arc (240:180:{4*\x})
arc (0:-180:2*\x and \x) %this is the wrap around up top
arc (0:-60:4*\x) arc (120:180:6*\x);
\draw[dashed] (0,0) arc (180:0:2*\x and \x);
\draw[dashed] (10*\x,0) arc (180:0:2*\x and \x);
\draw[dashed] (5*\x,{\x*sin(60)*10}) arc (180:0:2*\x and \x);
\draw (5*\x,{\x*sin(60)*10}) arc (180:0:2*\x);
\draw[red] (2*\x,-\x) arc (180:0:5*\x and 6*\x);
}
= \epsilon^r\ev_y\epsilon^\ell(\Tr(y^*)\circ\Tr(y))
\]
is a non-degenerate pairing between $\Tr(y)$ and $\Tr(y^*)$.
\end{cor}
\begin{proof}
The adjoint of a composite is the composite of adjoints, with witnesses being the composite of the witnesses for the adjunction, so this follows from Proposition \ref{prop:biadjoint}. Pictorially, the composite on the left decomposes as
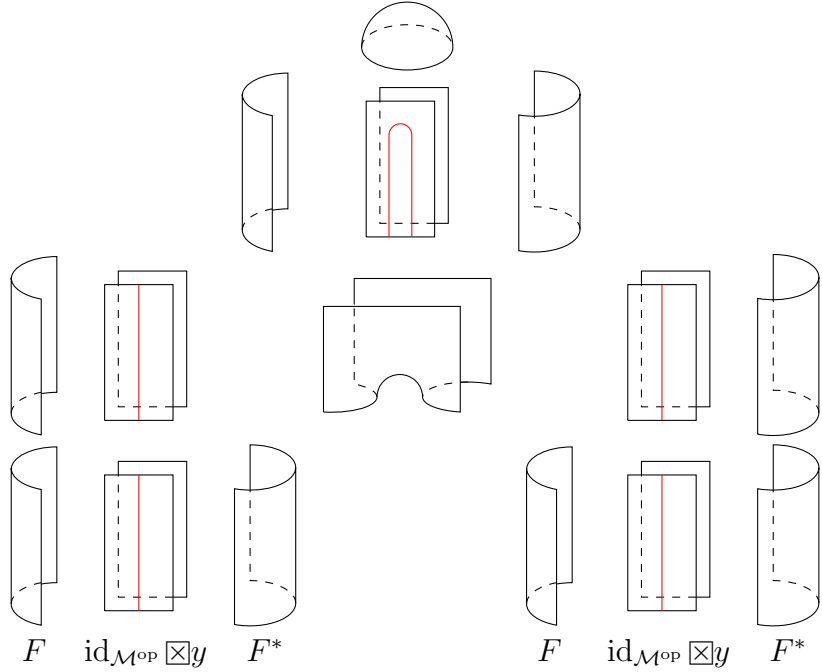
\begin{center}
\begin{tabular}{ccccccc}
& & &
$\tikzmath{
\draw (0,0) arc (-180:0:2*\x and \x) arc (0:180:2*\x);
\draw[dashed] (0,0) arc (180:0:2*\x and \x);
}$
\\
& & $\lefttube{250}{90}{\x}{\h}{2}$ &
$\tikzmath{
\draw (0,0) -- +(0,\h*\x) -- +(\h*\x/2,\h*\x) -- +(\h*\x/2,0) -- (0,0);
\draw[dashed] (\h*\x/2,\h*\x/10) -- (\h*\x/10,\h*\x/10) -- (\h*\x/10,\h*\x);
\draw (\h*\x/10,\h*\x) -- +(0,\h*\x/10) -- +(\h*\x/2,\h*\x/10) -- (\h*\x*0.6,\h*\x/10) -- +(-\h*\x/10,0);
\draw[red] (\h*\x/6,0) -- +(0,\h*\x*.75) arc (180:0:\h*\x/12) -- (2*\h*\x/6,0);
}$ &
$\righttube{250}{90}{\x}{\h}{2}$
\\
$\lefttube{250}{90}{\x}{\h}{2}$ &
$\tikzmath{
\draw (0,0) -- +(0,\h*\x) -- +(\h*\x/2,\h*\x) -- +(\h*\x/2,0) -- (0,0);
\draw[dashed] (\h*\x/2,\h*\x/10) -- (\h*\x/10,\h*\x/10) -- (\h*\x/10,\h*\x);
\draw (\h*\x/10,\h*\x) -- +(0,\h*\x/10) -- +(\h*\x/2,\h*\x/10) -- (\h*\x*0.6,\h*\x/10) -- +(-\h*\x/10,0);
\draw[red] (\h*\x/4,0) -- +(0,\h*\x);
}$ & &
$\saddle{260}{60}{0.2}{7}{9}{3}$ & &
$\tikzmath{
\draw (0,0) -- +(0,\h*\x) -- +(\h*\x/2,\h*\x) -- +(\h*\x/2,0) -- (0,0);
\draw[dashed] (\h*\x/2,\h*\x/10) -- (\h*\x/10,\h*\x/10) -- (\h*\x/10,\h*\x);
\draw (\h*\x/10,\h*\x) -- +(0,\h*\x/10) -- +(\h*\x/2,\h*\x/10) -- (\h*\x*0.6,\h*\x/10) -- +(-\h*\x/10,0);
\draw[red] (\h*\x/4,0) -- +(0,\h*\x);
}$ &
$\righttube{250}{90}{\x}{\h}{2}$
\\[30pt]
$\lefttube{250}{90}{\x}{\h}{2}$ &
$\tikzmath{
\draw (0,0) -- +(0,\h*\x) -- +(\h*\x/2,\h*\x) -- +(\h*\x/2,0) -- (0,0);
\draw[dashed] (\h*\x/2,\h*\x/10) -- (\h*\x/10,\h*\x/10) -- (\h*\x/10,\h*\x);
\draw (\h*\x/10,\h*\x) -- +(0,\h*\x/10) -- +(\h*\x/2,\h*\x/10) -- (\h*\x*0.6,\h*\x/10) -- +(-\h*\x/10,0);
\draw[red] (\h*\x/4,0) -- +(0,\h*\x);
}$ &
$\righttube{250}{90}{\x}{\h}{2}$ & &
$\lefttube{250}{90}{\x}{\h}{2}$ &
$\tikzmath{
\draw (0,0) -- +(0,\h*\x) -- +(\h*\x/2,\h*\x) -- +(\h*\x/2,0) -- (0,0);
\draw[dashed] (\h*\x/2,\h*\x/10) -- (\h*\x/10,\h*\x/10) -- (\h*\x/10,\h*\x);
\draw (\h*\x/10,\h*\x) -- +(0,\h*\x/10) -- +(\h*\x/2,\h*\x/10) -- (\h*\x*0.6,\h*\x/10) -- +(-\h*\x/10,0);
\draw[red] (\h*\x/4,0) -- +(0,\h*\x);
}$ &
$\righttube{250}{90}{\x}{\h}{2}$
\\[30pt]
$F$ & $\id_{\cM^\op} \boxtimes y$ & $F^*$ & & $F$ & $\id_{\cM^\op} \boxtimes y$ & $F^*$
\end{tabular}
\end{center}
and notice that the three central maps stacked in the middle are all evaluations.
\end{proof}

\bibliographystyle{alpha}
\bibliography{bibliography.bib}

\end{document}